\documentclass[11pt]{amsart}
\usepackage{amssymb}
\usepackage{amsmath}
\usepackage[all]{xy}
\usepackage{enumerate}
\usepackage[usenames,dvipsnames]{pstricks}
\usepackage{epsfig}
\usepackage{pst-grad} 
\usepackage{pst-plot} 
\usepackage{graphicx}
\usepackage{hyperref}

\psset{unit=1pt}
\psset{arrowsize=4pt 1}
\psset{linewidth=.5pt}

\renewcommand{\phi}{\varphi}

\renewcommand{\tilde}{\widetilde}
\renewcommand{\hat}{\widehat}

\newcommand{\Q}{\mathbb{Q}}
\newcommand{\R}{\mathbb{R}}
\newcommand{\N}{\mathbb{N}}
\newcommand{\Z}{\mathbb{Z}}

\newcommand{\A}{\mathbb{A}}

\DeclareMathOperator{\BOX}{Box}

\newtheorem{theorem}{Theorem}[section]
\newtheorem{lemma}[theorem]{Lemma}

\newtheorem{corollary}[theorem]{Corollary}

\newtheorem{question}[subsection]{Question}

\theoremstyle{definition}

\newtheorem{remark}[theorem]{Remark}
\newtheorem{example}[theorem]{Example}

\begin{document}

\title[Inequalities in Ehrhart theory]{
Additive number theory and \\
inequalities in Ehrhart theory} 
\author{Alan Stapledon}
\address{Department of Mathematics\\University of British Columbia\\ BC, Canada V6T 1Z2}
\email{astapldn@math.ubc.ca}

\date{}
\thanks{The original version of this work was completed while the author was a graduate student at the 
University of Michigan. The author is grateful to Jeff Lagarias for some amazing insights and for bringing Kneser's theorem to his attention. He would like to thank Christian Haase, 
Benjamin Nill and Sam Payne 
for some stimulating conversations
%
during  a great week at FU Berlin in January, 2009, and would like to thank Jozsef Solymosi for pointing him to the work of Lev.} 


\begin{abstract}
We introduce a powerful connection between Ehrhart theory and additive number theory, and use it 
to produce infinitely many new classes of inequalities between the coefficients of the $h^*$-polynomial of a lattice polytope. This greatly improves upon the three known classes of inequalities, which 
were proved  using techniques from commutative algebra and combinatorics.  As an application, we deduce all possible `balanced' inequalities between the  coefficients of the $h^*$-polynomial  of a lattice polytope containing an interior lattice point, in dimension at most $6$.

\end{abstract}

\maketitle

\section{Introduction}

Ehrhart theory concerns the enumeration of lattice points in dilations of a lattice polytope. 
More precisely, if $P \subseteq \R^{d}$ is a $d$-dimensional polytope with integer vertices
and $f_{P}(m)$ denotes the number of integer-valued points in the $m$'th dilate of $P$, then a famous theorem of Ehrhart \cite{ehrhartpolynomial} asserts that $f_{P}(m)$ is a polynomial in $m$ of degree $d$, called the \emph{Ehrhart polynomial} of $P$. 
The central long-standing open problem in the field is the following:

\begin{question}\label{biggie}
Can one characterize all polynomials which can be interpreted as the Ehrhart polynomial of some lattice polytope?
\end{question}

The dimension $2$ case of Question~\ref{biggie} was settled by Scott in 1976 \cite{ScoConvex}, and the higher dimensional cases remain wide open. Although a complete answer to this problem could well be impossibly difficult, the goal of this paper is to introduce a new approach to Ehrhart theory which restricts the possible polynomials appearing as Ehrhart polynomials. 
More specifically, we introduce a connection 
between Ehrhart theory and  additive number theory which produces infinitely many new classes of inequalities satisfied by the coefficients of the Ehrhart polynomial. 

We remark that an answer to Question~\ref{biggie} would also be interesting from a geometric perspective. 
On the one hand, it would give a characterization of all Hilbert polynomials of polarized, projective toric varieties (Section 4.4 in \cite{FulIntroduction}), while, on the other hand, it would characterize the possible dimensions of orbifold cohomology of crepant partial resolutions of Gorenstein toric singularities (Theorem 4.6 in \cite{YoWeightI}, Introduction in \cite{KarEhrhart}). In particular, from these two perspectives, our results can be interpreted geometrically.

One may rephrase Question~\ref{biggie} by considering the corresponding generating series of the Ehrhart polynomial. More precisely, one may write 
\[
\sum_{m \geq 0} f_{P}(m)t^{m} = \frac{h^*(t)}{(1 - t)^{d + 1}},
\]
where $h^*(t) = h^*_{0} + h^*_{1}t + \cdots + h^*_{d}t^{d}$ is a polynomial of degree at most $d$ with integer coefficients, called the \emph{$h^*$-polynomial} of $P$; alternative names in the literature include $\delta$-polynomial and Ehrhart $h$-polynomial. In what follows, we will often identify the $h^*$-polynomial with its vector of coefficients $( h^*_{0} , h^*_{1}, \ldots, h^*_{d}    )$. 
The \emph{degree}
 $s$ of $P$ is defined to be the degree of  $h^*(t)$, 
 and the \emph{codegree} $l$ of $P$ is defined by $l = d + 1 - s$. 
Observe that $h^*(t)$ and $f_P(m)$ encode equivalent information, and, hence, 
Question~\ref{biggie} is equivalent to asking for a characterization of all polynomials which appear as the $h^*$-polynomial of some lattice polytope. 
Although there are almost no general results in this direction, in 1984, in their ground-breaking paper \cite{BMLattice}, Betke and McMullen proved that some linear inequalities exist between the coefficients of the 
$h^*$-polynomial, and 
challenged the mathematical community to find better inequalities.
A great deal of work has been done on this problem since that time, and we summarize the current state of knowledge below.

It follows from Ehrhart's original results that $h^*_0 = 1$ and $ 0 \le h^*_d \le  h^*_1$ (see, for example, \cite{YoInequalities}). The first deep result was Stanley's proof of the non-negativity of the coefficients  $h^*_i$ in \cite{StaDecompositions}. Stanley's proof used commutative algebra and the theory of Cohen Macauley rings, while a combinatorial proof was later given by Betke and McMullen in \cite{BMLattice}.
Using techniques of commutative algebra, Hibi proved in \cite{HibSome} that 
\begin{equation}\label{ineq1}
h^*_{d} + h^*_{d - 1} + \cdots + h^*_{d - i} \le h^*_{0} + h^*_{1} 
+ \cdots  + h^*_{i + 1}
 \textrm{ for }     i = 0 , \ldots, \lfloor \frac{d}{2}     \rfloor - 1,
\end{equation}
while Stanley proved in  \cite{StaHilbert1} that 
\begin{equation}\label{ineq2}
h^*_{0} + h^*_{1} + \cdots + h^*_{i} \leq h^*_{s} + h^*_{s - 1} + \cdots + h^*_{s - i}
\textrm{ for }     i = 0 , \ldots, \lfloor \frac{s}{2} \rfloor.
\end{equation}
In \cite{HibLower}, Hibi used combinatorial techniques  to prove that 
\begin{equation}\label{ineq3}
\textrm{ if } h^*_{d} \neq 0 \textrm{ then } 1 \leq h^*_{1} \leq h^*_{i} \textrm{ for }     i = 2 , \ldots, d - 1.
\end{equation}
Until recently, all known inequalities between the coefficients of the $h^*$-polynomial could be deduced from these three classes of inequalities. 

A new combinatorial approach to $h^*$-polynomials was introduced by the author in \cite{YoInequalities}. More specifically, motivated by weighted Ehrhart theory \cite{YoWeightI}, 
the author introduced natural polynomials $a(t)$ and $b(t)$ which encode the $h^*$-polynomial, and are defined by
\begin{equation}\label{a}
a_{i + 1} = h^*_{0} + \cdots + h^*_{i + 1} - h^*_{d} - \cdots - h^*_{d - i} \, \textrm{ for } - 1 \le i \le d - 1, 
\end{equation}
\begin{equation}\label{b}
b_{i} = -h^*_{0} - \cdots - h^*_{i} + h^*_{s} + \cdots + h^*_{s - i} \, \textrm{ for } 0 \le i \le s - 1.
\end{equation}
One verifies (Lemma 2.3 in \cite{YoInequalities}) that $a(t) = t^d a(t^{-1})$, $b(t) = t^{s - 1}b(t^{-1})$, and 
\begin{equation}\label{niceshirt}
(1 + t + \cdots + t^{l - 1})h^*(t) = a(t) + t^{l}b(t).
\end{equation}
Observe that Hibi and Stanley's inequalities \eqref{ineq1} and \eqref{ineq2} are equivalent
to the non-negativity of the coefficients of $a(t)$ and $b(t)$, respectively. In fact, the author proved that the coefficients of $a(t)$ are positive, and the following result may be viewed, along with \eqref{ineq3}, as the current state of knowledge. 

\begin{theorem}\cite[Theorem 2.14]{YoInequalities}\label{previous}
In the decomposition $(1 + t + \cdots + t^{l - 1})h^*(t) = a(t) + t^{l}b(t)$ above, the coefficients $b_i$ are non-negative and
$1 = a_{0} \leq a_{1} \leq a_{i}$
for $i = 2, \ldots, d - 1$.
\end{theorem}

We remark that Hibi's inequality \eqref{ineq3} would follow if we knew that $b_0 \le b_i$ for $0 \le i \le d - 1$, when $s = d$ and $l = 1$. This stronger inequality is proved later in Theorem~\ref{refinement}. 
We also note that the inequalities $a_1 \le a_i$  in Theorem~\ref{previous} are equivalent to the following refinement of \eqref{ineq1}, which was suggested, without proof, by Hibi in \cite{HibLower}.
\begin{equation}
h^*_{d - 1} + \cdots + h^*_{d - i}  \le h^*_{2} + \cdots 
+ h^*_{i + 1} \textrm{ for }     i = 0 , \ldots, \lfloor d/2     \rfloor - 1.
\end{equation}

In this paper, we 
prove infinitely many new classes of inequalities  between the coefficients of the $h^*$-polynomial. In fact, these inequalities are difficult to express in terms of the coefficients of the $h^*$-polynomial, but have a natural form in terms of the polynomials $a(t)$ and $b(t)$.  This result  is achieved by creating a highly technical, but extremely powerful, connection with additive number theory. In particular, one of our main tools is the following deep result of  Kemperman and Scherk. We refer the reader to \cite{LevRestricted} for a history of results in additive number theory.

\begin{theorem}[\cite{MSAdvanced}, \cite{KemComplexes}, \cite{KemSmall}]      
If $A$ and $B$ are finite  subsets of an abelian group and $A \cap (-B) = \{ 0 \}$, then 
\[
|A + B| \ge |A| + |B| - 1.
\]
\end{theorem}

In order to state our main results, we first introduce some notation. If $0 \le r \le r'$, then let 
$Q(r,r') \subseteq \R^{r + r' + 1}$ be the rational polyhedron defined by all $(r + r' + 1)$-tuples
$(x_0, x_1, \ldots, x_{r + r'})$ of non-negative real numbers satisfying: 
\begin{enumerate}
\item\label{cat1} $x_i \ge 1$ \;  for \;  $0 \le i \le r$, 
\item\label{cat2}  $x_i \ge \frac{r + 1}{2i + 1}$ \;  for \;  $r + 1 \le  i \le \lfloor \frac{r + r'}{2} \rfloor$,
\item\label{cat3} $x_{i} + \cdots + x_{2r - i} \ge 2r - 2i + 1$ \; for \;  
$0 \le  i \le r - 1$,
\item\label{cat4} $x_{i} + \cdots + x_{r + r' - i} \ge r  + 1 - 2i\frac{r + 1}{r + r' + 1}$ \, for \, 
$r  + 1 \le i < \frac{r + r'}{2}$. 
\end{enumerate}

\begin{example}\label{swim}
Consider the case when  $r = r'$ and $Q(r,r) \subseteq \R^{2r + 1}$. 
If we let $e_0, \ldots, e_{2r}$ denote the standard basis of $\R^{2r + 1}$, then one verifies that the vertices of $Q(r,r)$ are given by 
\[
\{ e_0 + \cdots + e_r + \sum_{i = 0}^{r - 1} e_{k_i}  \mid i \le k_i \le 2r - i 
\}.
\]
\end{example}

\begin{example}\label{swum}
Consider the case when  $r = 0$ and $Q(0,r') \subseteq \R^{r' + 1}$. 
If we let $e_0, \ldots, e_{r'}$ denote the standard basis of $\R^{r' + 1}$, then one verifies that the vertices of $Q(0,r')$ are given by 
\[
\{ \sum_{i = 0}^{\lfloor \frac{r'}{2} \rfloor} \frac{1}{2i + 1} e_{i} + 
\sum_{i = \lceil \frac{r'}{4}  \rceil}^{\lfloor \frac{r' - 1}{2} \rfloor}
( \frac{2}{r' + 1} - \frac{1}{2i + 1})e_{k_i} \mid i \le k_i \le r' - i 
\}.
\]

\end{example}

Our first main result is the following theorem. Although the statement is a little technical, we present 
simpler consequences and examples below. 

\begin{theorem}\label{superA}
With the notation above, if $0 \le r \le r'$,  
$d \ge 2r' + r + 7$, 
and $(\lambda_0, 
\ldots, \lambda_{r + r'})$ is a vertex of $Q(r,r')$, then 
\begin{equation*}
\lambda  a_1 +  \sum_{j = 0}^{r} a_{j + 2}  \le  a_{r' + 3} + 
\sum_{j = 0}^{r + r'} \lambda_j a_{r' + 4 + j}, 
\end{equation*}
where $\lambda = \sum_{j = 0}^{r + r'} \lambda_j - r$. 
\end{theorem}



The following corollary, although a little less optimal, may be useful in practice. 

\begin{corollary}\label{corA}
If $0 \le r \le r'$,  
$d \ge 2r' + m + 7$ and $m = \max( 2r, \lfloor  \frac{r  + r'}{2} \rfloor)$,  then
\begin{equation*}
(m - r + 1)a_1 + a_2 + \cdots + a_{r + 2} \le a_{r' + 3} + a_{r' + 4} +  \cdots + a_{r' + 4 + m}.
\end{equation*}
Equivalently,
\[
\sum_{j = 0}^r (m - r + 1 + j) h^*_{d - 1 - j} + (m + 2) \sum_{j = 0}^{r' - r} h^*_{d -r - 2 - j} + 
\sum_{j = 0}^m (m + 1 - j) h^*_{d - r' - 3 - j} 
\]
\[
\le \sum_{j = 0}^r (m - r + 1 + j) h^*_{j + 2} + (m + 2) \sum_{j = 0}^{r' - r} h^*_{r + 3 + j} + 
\sum_{j = 0}^m (m + 1 - j) h^*_{r' + 4 + j}.  
\]
\end{corollary}

\begin{remark}
As Corollary~\ref{corA} demonstrates, it will be more natural to state results in terms of the polynomials $a(t)$ and $b(t)$, rather than $h^*(t)$. On the other hand, one easily verifies that 
$(1 + t + \cdots + t^{l - 1})h^*(t) = a(t) + t^{l}b(t)$ is the unique decomposition of $(1 + t + \cdots + t^{l - 1})h^*(t)$ as a sum of polynomials satisfying 
$a(t) = t^d a(t^{-1})$ and  $b(t) = t^{s - 1}b(t^{-1})$. Hence, in practice, it is very easy to compute $a(t)$ and $b(t)$ from $h^*(t)$. 
\end{remark}

\begin{example}\label{example1}
Setting $r = r' = 0$ in the above corollary implies that for $d \ge 7$, 
\[
a_1 + a_2 \le a_3 + a_4,
\]
or, equivalently, 
\[
h^*_{d - 1} + 2 h^*_{d - 2} + h^*_{d - 3}  \le h^*_{2} + 2 h^*_{3} + h^*_{4}.
\]
We claim that the vector $v  = (1,2,2,1,2,2,1,0)$ can not be realized as the coefficients of the  $h^*$-polynomial of a lattice polytope. In this case, $a(t) = (1 + t)h^*(t) = 1 + 3t + 4t^2 + 3t^3 + 3t^4 + 4t^5 + 3t^6 + t^7$ and $b(t) = 0$. On the one hand, the coefficients of $b(t)$ are non-negative and $1 = a_{0} \leq a_{1} \leq a_{i}$
for $2 \le i \le 6$, so the vector $v$ satisfies all previously known inequalities (see Theorem~\ref{previous}). On the other hand, 
$a_1 + a_2 = 7 > 6 = a_3 + a_4$, so $v$ violates the above inequality with $d = 7$. 
\end{example}

\begin{example}\label{equal}
If we set $r = r'$ in Theorem \ref{superA}, then, using Example~\ref{swim}, we get an explicit description of the inequalities in the theorem: 
\[
(r + 1)a_1 + a_2 + \cdots + a_{r + 2} \le a_{r + 3} + a_{r + 4} + \cdots + a_{2r + 4} + \sum_{i = 0}^{r - 1} a_{r + 4 + k_i}, 
\]
for some $i \le k_i \le 2r - i$, 
and for $d \ge 3r + 7$. 
\end{example}

\begin{example}\label{zero}
If we set $r = 0$ in Theorem \ref{superA}, then, using Example~\ref{swum}, we also get an explicit description of the inequalities in the theorem: 
\[
\lambda a_1 + a_2 \le a_{r' + 3} + \sum_{i = 0}^{\lfloor \frac{r'}{2} \rfloor} \frac{1}{2i + 1} a_{r' + 4 + i} + 
\sum_{i = \lceil \frac{r'}{4}  \rceil}^{\lfloor \frac{r' - 1}{2} \rfloor}
( \frac{2}{r' + 1} - \frac{1}{2i + 1})a_{r' + 4 + k_i}, 
\]
for some $i \le k_i \le r' - i$, 
and for $d \ge 2r' + 7$, where $\lambda =  
\sum_{i = 0}^{\lfloor \frac{r'}{2} \rfloor} \frac{1}{2i + 1} + 
\sum_{i = \lceil \frac{r'}{4}  \rceil}^{\lfloor \frac{r' - 1}{2} \rfloor} ( \frac{2}{r' + 1} - \frac{1}{2i + 1})$.
\end{example}

Using the previous two examples and some additional computations,  in Figure~\ref{noint} we compute all the inequalities from Theorem~\ref{superA} in dimension less than or equal to $12$. 

\begin{remark}
The symmetry of $a(t)$ implies that the lower bound $d \ge 2r' + r + 7$ can not be lowered, and, as demonstrated in Figure~\ref{noint}, that it needs to be raised for particular inequalities in order to produce new examples. 
\end{remark}

\begin{figure}[htb]

\begin{tabular}{ | l |  l | l |}
\hline
Inequality & $(r,r')$ & 
Dimension 
\\
\hline
$a_1 + a_2 \le a_3 + a_4$  & $(0,0)$ &  $d \ge 7$ \\
$a_1 + a_2 \le a_4 + a_5$  & $(0,1)$ & $d \ge  9$\\
    $2a_1 + a_2 + a_3 \le a_4 + a_5 + 2a_6$ & $(1,1)$ & $d \ge  10$ \\ 
   $2a_1 + a_2 + a_3 \le a_4 + 2a_5 + a_6$    & $(1,1)$ & $d \ge  10$ \\ 
     $2a_1 + a_2 + a_3 \le a_4 + a_5 + a_6 + a_7$  & $(1,1)$ & $d \ge  11$ \\ 
    $\frac{4}{3}a_1 + a_2 \le a_5 + a_6 + \frac{1}{3} a_7$ & $(0,2)$ & $d \ge  11$ \\ 
$2a_1 + a_2 + a_3 \le a_5 + a_6 + 2a_7$ & $(1,2)$ & $d \ge 12$ \\
$2a_1 + a_2 + a_3 \le a_5 + 2a_6 + a_7$  & $(1,2)$ & $d \ge 12$ \\
\hline 
\end{tabular}

 \caption{Inequalities from Theorem~\ref{superA} in 
 dimension at most $12$.}
        \label{noint}
        \end{figure}

For the remainder of the introduction, we will specialize and only consider the case when $h^*_d \ne 0$, or, equivalently, when $s = d$ and $l  = 1$. It follows from Ehrhart's original results that $h^*_d$ can be interpreted as the number of interior lattice points in $P$ (see, for example, \cite{BRComputing}), and hence our  assumption is that $P$ contains an interior lattice point. In this case, the decomposition 
\[
h^*(t) = a(t) + tb(t), 
\]
was first considered by Betke and McMullen in \cite{BMLattice}. Using techniques
 of Hibi \cite{HibLower}, Betke and McMullen \cite{BMLattice}, and the author \cite{YoInequalities},  we give an explicit description of this decomposition in Theorem~\ref{main}, 
 and deduce the following theorem as a corollary, which, on the one hand, refines \eqref{ineq1} 
 and \eqref{ineq2} when $s = d$, and, on the other hand, includes \eqref{ineq3} as a consequence. 

\begin{theorem}\label{refinement}
If $P$ contains an interior lattice point and $h^*(t) = a(t) + tb(t)$ is the decomposition above, 
then the coefficients of $a(t)$ and $b(t)$ satisfy:  
\[1 = a_{0} \leq a_{1} \leq a_{i} \;  \textrm{ for } 2 \leq i \leq d - 1,\]
\[0 \leq b_{0} \leq b_{i}  \; \textrm{ for } 1 \leq i \leq d - 2.\] 
Equivalently, the coefficients of the $h^*$-polynomial of $P$ satisfy:
\begin{equation*}
1  = h^*_{0} \leq h^*_{d} \leq h^*_{1}, 
\end{equation*}
\begin{equation*}
h^*_{1} + \cdots + h^*_{i} \le
 h^*_{d - 1} + \cdots + h^*_{d - i} \le h^*_{2} + \cdots
+ h^*_{i + 1},
\end{equation*}
for $i = 1, \ldots, \lfloor \frac{d - 1}{2} \rfloor$.
\end{theorem}

\begin{example}
We claim that the vector $v  = (1,2,3,2,2,2)$ can not be realized as the coefficients of the $h^*$-polynomial of a lattice polytope. In this case, $a(t) = 1 + t + 2t^2 + 2t^3 + t^4 + t^5$ and $b(t) = 1 + t + t^3 + t^4$. On the one hand, the coefficients of $b(t)$ are non-negative and $1 = a_{0} \leq a_{1} \leq a_{i}$
for $2 \le i \le 4$, so the vector $v$ satisfies the inequalities in Theorem~\ref{previous}. Moreover,
$h^*_1 = 2 \le h^*_i$ for $2 \le i \le 4$, so $v$ satisfies Hibi's inequality \eqref{ineq3}, and hence
all previously known inequalities. On the other hand, $b_0 = 1 > b_2 = 0$, so $v$ does not satisfy the inequalities in Theorem~\ref{refinement} with $d = 5$. 

\end{example}

\begin{remark}
The corresponding theorem is false if $P$ does not contain an interior lattice point (Example 2.4 in \cite{YoInequalities}, essentially due to Henk and Tagami). 
\end{remark}

\begin{remark}
For $d \le 4$, one verifies that Theorem~\ref{refinement} is equivalent to Theorem~\ref{previous} and Hibi's inequality \eqref{ineq3}. 
\end{remark}

In fact, using examples of Payne \cite{PayEhrhart}, we prove the following theorem, which says that the inequalities in Theorem~\ref{refinement} give all the inequalities  of a certain type in dimension at most $5$. We consider the cases of dimensions $6$ and $7$ later in the introduction. 
More precisely, we say that a linear inequality $\sum_{i = 0}^{d} \alpha_i h^*_i \ge 0$ is \emph{balanced} if 
$\sum_{i = 0}^d \alpha_i = 0$. Note that all known (minimal) inequalities are balanced, and, in fact, all inequalities produced using the techniques of this paper will be balanced. 

\begin{theorem}\label{all5}
Every balanced inequality  for $h^*$-polynomials of polytopes containing an interior lattice point 
follows from the inequalities in Theorem~\ref{refinement} if and only if $d \le 5$.

\end{theorem}


We now return to our connection with additive number theory. In the case when $P$ contains an interior lattice point, Kemperman and Scherk's theorem provides a powerful tool for deducing inequalities. 
As before, the statement of the theorem is a little technical, but we will provide corollaries and examples below. 
Recall that 
for $0 \le r \le r'$,  
$Q(r,r') \subseteq \R^{r + r' + 1}$ is the rational polyhedron defined by all $(r + r' + 1)$-tuples
$(x_0, x_1, \ldots, x_{r + r'})$ of non-negative real numbers satisfying conditions \eqref{cat1}, \eqref{cat2},
\eqref{cat3} and \eqref{cat4}, listed previously.  If $r < 0$ and $r + r' + 1 \ge 0$, then we consider $Q(r,r')$ to be the origin in $\R^{r + r' + 1}$.

\begin{theorem}\label{variant}
Let $P$ be a $d$-dimensional lattice polytope containing an interior lattice point, and let $0 \le r \le r'$ and $0 \le \alpha \le r + 1$. 
With the notation above, if $(\lambda_0, 
\ldots, \lambda_{r + r'})$ and $(\mu_0, 
\ldots, \mu_{r + r'})$ are  vertices of $Q(r,r')$, and $(\lambda_0', \ldots, \lambda_{r + r' - 2 \alpha}')$ is a vertex of $Q(r - \alpha, r' - \alpha)$, then
 

\begin{enumerate}

\item\label{try1} for $d \ge 2r' + r + 7$ and   $0 \le \alpha \le r$,
\[
\lambda  a_1 + \mu b_0 +  \sum_{j = 0}^{r} a_{j + 2}  + \sum_{j = 0}^{r - \alpha} b_{j + 1}  \le  a_{r' + 3} + 
\sum_{j = 0}^{r + r'} \lambda_j a_{r' + 4 + j} +  b_{r' + 2 - \alpha} + \sum_{j = 0}^{r + r'} \mu_j b_{r' + 3 - \alpha + j}, 
\]
where $\lambda = \sum_{j = 0}^{r + r'} \lambda_j - r$ and $\mu = \sum_{j = 0}^{r + r' } \mu_j - r + \alpha$,

\item\label{try2} for $d \ge 2r' + r + 6$ and  $r > 0$,
\begin{equation*}
\lambda  a_1 + \mu b_0 +  \sum_{j = 0}^{r} a_{j + 1}  + \sum_{j = 0}^{r} b_{j + 1}  \le  a_{r' + 2} + b_{r' + 2} + 
\sum_{j = 0}^{r + r'} \lambda_j a_{r' + 3 + j} + \sum_{j = 0}^{r + r'} \mu_j b_{r' + 3 + j}, 
\end{equation*}
where $\lambda = \sum_{j = 0}^{r + r' } \lambda_j - r$ and $\mu = \sum_{j = 0}^{r + r' } \mu_j - r$,

\item\label{try3} for $d \ge 2r' + r  + \alpha + 6$, 
\begin{equation*}
\lambda'  a_1 + \mu b_0  + \sum_{j = 0}^{r} b_{j + 1}  \le  \sum_{j = 0}^{\alpha} b_{r' + 2 + j} + 
\sum_{j = 0}^{r + r' - 2 \alpha} \lambda_j' a_{r' + \alpha +  3 + j} + \sum_{j = 0}^{r + r'} \mu_j b_{r' + \alpha + 3 + j}, 
\end{equation*}
where $\lambda' = \sum_{j = 0}^{r + r' - 2\alpha} \lambda_j'$ and $\mu = \sum_{j = 0}^{r + r' } \mu_j - r + \alpha$.

\end{enumerate}







\end{theorem}

As before, the following corollary may be useful in practice. 

\begin{corollary}\label{dos}
If $P$ is a $d$-dimensional lattice polytope containing an interior lattice point, and $0 \le r \le r'$ 
 and $m =  \max( 2r, \lfloor  \frac{r  + r'}{2} \rfloor)$, then

 

\begin{enumerate}

\item for $d \ge 2r' + m + 7$ and   $0 \le \alpha \le r$,
\[
(m - r + 1) a_1 + (m - r + \alpha + 1) b_0 +  \sum_{j = 0}^{r} a_{j + 2}  + \sum_{j = 0}^{r - \alpha} b_{j + 1}  \le 
\sum_{j = 0}^{m + 1} a_{r' + 3 + j} 
+ \sum_{j = 0}^{m + 1} b_{r' + 2 - \alpha + j},
\]

\item for $d \ge 2r' + m + 6$ and  $r > 0$,
\begin{equation*}
(m - r + 1) (a_1 +  b_0) +  \sum_{j = 0}^{r} a_{j + 1}  + \sum_{j = 0}^{r} b_{j + 1}  \le  
\sum_{j = 0}^{m + 1} a_{r' + 2 + j} + \sum_{j = 0}^{m + 1} b_{r' + 2 + j},
\end{equation*}

\item for $d \ge 2r' + m  + \alpha + 6$ and $0 \le \alpha \le r + 1$, 
\begin{equation*}
(m' + 1)a_1 + (m + \alpha - r + 1)b_0  + \sum_{j = 0}^{r} b_{j + 1}  \le  
\sum_{j = 0}^{m'}  a_{r' + \alpha +  3 + j} + \sum_{j = 0}^{m + \alpha + 1} b_{r' + 2 + j}, 
\end{equation*}
where $m' =  \max( 2r - 2\alpha, \lfloor  \frac{r  + r' - 2\alpha}{2} \rfloor)$ for $0 \le \alpha \le r$ and 
$m' = -1$ if $\alpha = r + 1$. 

\end{enumerate}

\end{corollary}

\begin{example}\label{example2}
Setting $\alpha = r = r' = 0$ in the third part of the above corollary implies that for $d \ge 6$, 
\[
a_1 + b_0 + b_1 \le a_3 + b_2 + b_3,
\]
or, equivalently, 
\[
h^*_{1} + h^*_2  \le h^*_{d - 3} + h^*_{d - 2}.
\]
We claim that the vector $v  = (1,1,2,1,1,2,1)$ can not be realized as the coefficients of the  $h^*$-polynomial of a lattice polytope. In this case, $a(t) = 1 + t + t^2 + t^3 + t^4 + t^5 + t^6$ and $b(t) = t + t^4$. On the one hand, $1 = a_{0} \leq a_{1} \leq a_{i}$
for $2 \le i  \le 5$, and $0 \le b_0 \le b_i$ for $1 \le i \le 4$, so the vector satisfies all the inequalities of Theorem~\ref{refinement}. On the other hand, $h^*_1 + h^*_2 = 3 > h^*_3 + h^*_4 = 2$, so $v$ violates the above inequality with $d = 6$. 
\end{example}

\begin{example}
As in Examples~\ref{equal} and \ref{zero}, Example~\ref{swim} and Example~\ref{swum} give explicit formulas  for the inequalities in Theorem~\ref{variant} in the cases when $r = r'$ and $r = 0$, respectively. 
\end{example}

Using the previous example,  in Figure~\ref{int} we compute all the inequalities from Theorem~\ref{variant} in dimension at most $8$. 

\begin{figure}[htb]

\begin{tabular}{ | l |  l | l | l |}
\hline
Inequality & $(\alpha, r,r')$& type & 
Dimension 
\\
\hline
$a_1 + b_0 + b_1 \le a_3 + b_2 + b_3$  & $(0,0,0)$ & \eqref{try3} &  $d \ge 6$ \\
$a_1 + a_2 +  b_0 + b_1 \le a_3 + a_4 + b_2 + b_3$  & $(0,0,0)$ & \eqref{try1} &  $d \ge 7$ \\
$2b_0 + b_1 \le b_2 + b_3 + b_4$  & $(1,0,0)$ & \eqref{try3} & $d \ge  7$\\
    $a_1 + b_0 + b_1 \le a_4 + b_3 + b_4$ & $(0,0,1)$ & \eqref{try3} & $d \ge  8$ \\ 
 
\hline 
\end{tabular}

 \caption{Inequalities from Theorem~\ref{variant} in 
 dimension at most $8$.}
        \label{int}
        \end{figure}



Recall that Theorem~\ref{all5} states that we know all balanced inequalities for $d \le 5$. 
In the case when $d = 6$, the above results give `essentially all' balanced 
inequalities.
We will make this statement more precise later in the introduction (Remark~\ref{noncon}), and limit ourselves for the moment to the following theorem, in which inequalities \eqref{take1} and \eqref{take2} follow from Theorem~\ref{refinement} and inequality \eqref{take3} follows from Example~\ref{example2}.  
We say that a linear inequality $\sum_{i = 1}^{d} \alpha_i h^*_i \ge 0$ is \emph{strictly balanced} if 
$\sum_{i = 1}^d \alpha_i = 0$. Note that the term $h^*_0 = 1$ does not appear in
the definition of 
 strictly balanced, as opposed to balanced.  

\begin{theorem}\label{six}
In dimension $6$, every strictly balanced inequality for $h^*$-polynomials of polytopes containing an interior lattice point can be deduced from the following inequalities:

\begin{enumerate}

\item\label{take1} $h^*_6 \le h^*_1 \le h^*_5 \le h^*_2$

\item\label{take2} $h^*_1 + h^*_2 \le h^*_4 + h^*_5 \le h^*_2 + h^*_3$

\item\label{take3} $h^*_1 + h^*_2 \le h^*_3 + h^*_4$.

\end{enumerate} 
\end{theorem}

In the case when dimension equals $7$, we summarize our results in Figure~\ref{seven} (after removing some redundancies), and include two conjectural inequalities. The conjectural inequalities would give `essentially all' balanced inequalities in dimension $7$.

\begin{figure}[htb]

\begin{tabular}{ | l |  l | l |}
\hline
$1 = h^*_0 \le h^*_7 \le h^*_1 \le h^*_6 \le h^*_2$  & $1 = a_0 \le a_1 \le a_2$, $0 \le b_0 \le b_1$ & Theorem~\ref{refinement}  \\
$h^*_1 + h^*_2 \le h^*_5 + h^*_6$  & $b_0 \le b_2$ & Theorem~\ref{refinement}    \\
  $h^*_1 + h^*_2 + h^*_3 \le h^*_4 + h^*_5 + h^*_6$  & $b_0 \le b_3$ & Theorem~\ref{refinement}   \\ 
    $h^*_1 + h^*_2 \le h^*_4 + h^*_5$ & $a_1 + b_0 + b_1 \le a_3 + b_2 + b_3$ &  Theorem~\ref{variant} \eqref{try3} \\ 
   $h^*_1 + h^*_2 \le h^*_3 + h^*_4$    & $a_1 + a_2 + b_0 + b_1 \le a_3 + a_4 + b_2 + b_3$ &  Theorem~\ref{variant} \eqref{try1} \\ 
    $2h^*_5 + h^*_6 \le h^*_2 + 2h^*_3$ & $a_1 + a_2 \le a_3 + a_4$ & Theorem~\ref{superA} \\
 $2h^*_1 + 3h^*_2 + h^*_3 \le h^*_4 + 3h^*_5 + 2h^*_6$ & $2b_0 + b_1 \le b_2 + b_3 + b_4$ & Theorem~\ref{variant} \eqref{try3} \\
$2 h^*_1 + 3 h^*_2 + 2h^*_3 \le 2h^*_4 + 4 h^*_5 + h^*_6$ & 
$\frac{1}{2}a_1 + b_0 + b_1 \le \frac{1}{2}a_2 + b_2 + b_3$  &  Conjecture \\
$4 h^*_1 + 7 h^*_2 + 2h^*_3 \le 4h^*_4 + 6 h^*_5 + 3h^*_6$   & 
$\frac{1}{4}a_1 + \frac{1}{4}a_2 + b_0 + b_1 \le \frac{1}{2}a_3 + b_2 + b_3$ &  Conjecture \\
\hline 
\end{tabular}

 \caption{Inequalities in 
 dimension $7$ for polytopes with an interior lattice point.}
        \label{seven}
        \end{figure}

\begin{example}\label{reflex}
A lattice polytope $P \subseteq \R^d$ is \emph{reflexive} if it contains the origin as its unique interior lattice point, and if there exists a piecewise $\Z$-linear function $\psi: \R^d \rightarrow \R$ such that 
$P =  \{ v \in \R^d \mid \psi(v) \le 1 \}$. A classical result of Hibi states that a 
lattice polytope $P$ is a translation of a reflexive polytope if and only if its $h^*$-polynomial has symmetric coefficients \cite{HibDual}. Observe that the latter condition is equivalent to requiring that $P$ has degree $d$ and satisfies $h^*(t) = a(t)$ and $b(t) = 0$.

The coefficients of the $h^*$-polynomial of a reflexive polytope are unimodal for $d \le 5$, and
Hibi conjectured that unimodality holds in general. 
Payne and Musta\c t\v a gave a counterexample 
in
\cite{MPEhrhart}, and
further counterexamples are given by Payne in all dimensions $d \geq 6$ in  \cite{PayEhrhart}. 

Our results give a complete description of the inequalities satisfied by the $h^*$-polynomial of a reflexive lattice polytope in dimension at most $6$. The results are summarized 
in Figure~\ref{reflexive}. In dimension $7$, our results show that every strictly balanced inequality
 satisfied by the $h^*$-polynomial $(1, h^*_1, h^*_2, h^*_3, h^*_3, h^*_2, h^*_1, 1)$ of a reflexive polytope follows from the two inequalities $h^*_1 \le h^*_2$ and $h^*_1 + h^*_2 \le 2h^*_3$. 


\begin{figure}[htb]

\begin{tabular}{| l  |  l  |   l  | }
\hline
$h^*$-polynomial  & Inequalities &
$d$ 
\\
\hline
 $(1,h^*_1,1)$  & $1 \le h^*_1$ & $2$ \\
  $(1, h^*_1, h^*_1, 1)$ &  $1 \le h^*_1$ & $3$ \\
$(1, h^*_1, h^*_2, h^*_1, 1)$  & $1 \le h^*_1 \le h^*_2$ & 4\\
   $(1, h^*_1, h^*_2, h^*_2, h^*_1, 1)$ & $1 \le h^*_1 \le h^*_2$ & 5 \\ 
 $(1, h^*_1, h^*_2, h^*_3, h^*_2, h^*_1, 1)$ & $1 \le h^*_1 \le h^*_2, h^*_3$ & 6 \\ 
\hline 
\end{tabular}

 \caption{Balanced inequalities  for $h^*$-polynomials of reflexive polytopes in 
 dimension  at most $6$. } 
        \label{reflexive}
        \end{figure}

\end{example}

\begin{remark}\label{noncon}
We observe that convexity plays no role in the proofs of the results above.
In particular, rather than considering a lattice polytope with an interior lattice point, one could 
consider a piecewise $\Z$-linear function $\psi: \R^d \rightarrow \R$ on a projective, rational fan $\triangle \subseteq \R^d$ and prove the same inequalities for $Q = \{ v \in \R^d \mid \psi(v) \le 1 \}$. Moreover, one could consider a polytopal complex $Q'$ with faces given by $\{ \sigma \cap Q \mid \sigma \in \triangle \}$, but allow the lattice structure to vary on the faces of $Q'$. That is, if $\sigma$ and $\tau$ are cones in $\triangle$, then we consider $\sigma \cap Q$ and $\tau \cap Q$ as lattice polytopes with respect to lattices $N$ and $N'$, respectively, and only require that $N$ and $N'$ agree along $\sigma \cap \tau$. 
Allowing these more general objects, in Section~\ref{last} we prove that the inequalities in Theorem~\ref{six} give all possible balanced inequalities. 

\end{remark}

We conclude the introduction with an outline of the contents of the paper. In Section~\ref{decom}, we recall  some notions from \cite{YoInequalities} and set notation for the first three sections. In Section~\ref{add}, we explore some consequences of  Kemperman and Scherk's theorem  and, in Section~\ref{proof1}, we prove Theorem~\ref{superA} and Corollary~\ref{corA}. In the remainder of the paper we assume that all lattice polytopes contain an interior lattice point. 
In Section~\ref{decomposition}, we set notation and prove Theorem~\ref{refinement}. In Section~\ref{revisit}, we extend the results of Section~\ref{add} for polytopes with an interior lattice point, and, in Section~\ref{proof2}, we prove Theorem~\ref{variant} and Corollary~\ref{dos}. In Section~\ref{last}, we compute examples and prove Theorem~\ref{all5} and Theorem~\ref{six}.  
Throughout the paper, we refer the reader to \cite{FulIntroduction} for the necessary background on remarks involving toric varieties.

\section{Preliminaries}\label{decom}

The goal of this section is to briefly recall some notions from \cite{YoInequalities} and set notation for the proof of Theorem~\ref{superA}. 

We fix a  $d$-dimensional lattice polytope $P$ in a lattice $N$ of rank $d$, with
$h^*$-polynomial $h^*(t)$. 
We will often identify the $h^*$-polynomial with its vector of coefficients $( h^*_{0} , h^*_{1}, \ldots, h^*_{d}    )$. 
Recall that the degree $s$ of $P$ is the degree of $h^*(t)$ and the codegree $l$ of $P$ is defined by $d + 1 = s + l$.
Let $a(t)$ and $b(t)$ denote the polynomials with coefficients given by 
\[
a_{i + 1} = h^*_{0} + \cdots + h^*_{i + 1} - h^*_{d} - \cdots - h^*_{d - i} \, \textrm{ for } - 1 \le i \le d - 1, 
\]
\[
b_{i} = -h^*_{0} - \cdots - h^*_{i} + h^*_{s} + \cdots + h^*_{s - i} \, \textrm{ for } 0 \le i \le s - 1,
\]
By Lemma 2.3 in \cite{YoInequalities}, $a(t) = t^d a(t^{-1})$, $b(t) = t^{s - 1}b(t^{-1})$ and
\[
(1 + t + \cdots + t^{l - 1})h^*(t) = a(t) + t^{l}b(t).
\] 

Our goal is to recall an explicit description of  the polynomial $a(t)$. We refer the reader to 
 \cite{YoInequalities} for a similar description of $b(t)$. 
Fix a regular, lattice triangulation $\mathcal{S}$ of the boundary $\partial P$ of $P$ which contains every lattice point in $\partial P$ as a vertex, and regard the empty face as a face of dimension $-1$. 
An $r$-dimensional lattice polytope $G$ in $N$ is called a \emph{lattice-free simplex} if $G$ contains exactly $r + 1$ lattice points (necessarily its vertices). 
Note that by construction, every face of $\mathcal{S}$ is a lattice-free simplex. 
If $F$ is a non-empty face of $\mathcal{S}$ with vertices $v_{1}, \ldots, v_{k}$,  then
let \[ \BOX(F) = \{  v \in N \times \Z \mid   v = \sum_{i = 1}^{k} \alpha_{i} (v_{i}, 1), \; 0 < \alpha_{i} < 1 \},
\] 
and set $\BOX(\emptyset) = \{ 0 \}$. 
Recall that the \emph{$h$-polynomial} of a face $F$ of $\mathcal{S}$ is defined by
\begin{equation*} 
h_{F}(t) = \sum_{F \subseteq G} t^{\dim G - \dim F} (1 - t)^{d - \dim G}.
\end{equation*}
We will often write $h_{\mathcal{S}}(t) = h_{\emptyset}(t)$. The following well-known lemma follows from Poincar\'e duality and the Hard Lefschetz theorem for projective toric varieties.
Recall that a vector  $(\lambda_0, \ldots, \lambda_r)$ with symmetric coefficients is \emph{unimodal} if $\lambda_0 \le \lambda_1 \le \cdots \le \lambda_{\lfloor  r/2 \rfloor}$.

\begin{lemma}\label{patriots}
Let $\mathcal{S}$ be a regular, lattice triangulation of $\partial P$. 
If $F$ is a face of $\mathcal{S}$, then the $h$-polynomial of $F$ 
is a polynomial of degree $d - 1 - \dim F$ with symmetric, unimodal integer coefficients.
\end{lemma}

The following interpretation of the polynomial $a(t)$ appears in the proof of Theorem~\ref{previous} in \cite{YoInequalities}. 

\begin{lemma}\label{nonlemma}
If $u: N \times \Z \rightarrow \Z$ denotes projection onto the second co-ordinate, then
\begin{equation*}
a(t) = \sum_{F \in \mathcal{S}} \sum_{w \in \BOX(F)} t^{u(w)} h_{F}(t). 
\end{equation*}
\end{lemma}

\section{Lattice-free simplices and additive number theory}\label{add}

The goal of this section is to use additive number theory to 
analyze the distribution of lattice points in the cone over a lattice-free simplex. 

Our main tool will be the following famous result in additive number theory 
due to  Kemperman and Scherk. The author would like to thank Jeff Lagarias for some brilliant insights  and, in particular, for bringing a related theorem to the author's attention.  

\begin{theorem}[\cite{MSAdvanced}, \cite{KemComplexes}, \cite{KemSmall}]\label{useful}
If $A$ and $B$ are finite  subsets of an abelian group and $A \cap (-B) = \{ 0 \}$, then \[
|A + B| \ge |A| + |B| - 1.
\]
\end{theorem}

We now fix our notation throughout this section. 
Recall from the previous section that a  $(d - 1)$-dimensional lattice polytope $G$ in $N$ is called a \emph{lattice-free simplex} if $G$ contains exactly $d$ lattice points (necessarily its vertices). 
We fix a $(d - 1)$-dimensional lattice-free simplex  $G$ with  vertices $v_{1}, \ldots, v_{d}$ and
let $C_G$ denote the cone over $G \times \{1\}$ in $N_{\R} \times \R$, where $N_{\R} = N \otimes_{\Z} \R$.  
Recall that if $F$ is a non-empty face of $G$ with vertices $v_{i_1}, \ldots,v_{i_{r}}$,  then 
\[ \BOX(F) = \{ v \in N \times \Z \mid
v = \sum_{j = 1}^{r} \alpha_j 
(v_{i_j}, 1) , \; 0 < \alpha_j < 1 \}, \] and $\BOX(\emptyset) = \{ 0 \}$. 
If $N(G)$ denotes the quotient of $N \times \Z$ by the sublattice generated by $(v_{1},1),\ldots, (v_{d},1)$, then $N(G)$ is a finite abelian group with elements in bijection with $\coprod_{F \subseteq G} \BOX(F)$, and we will often identify elements of $N(G)$ with their corresponding lattice points.

\begin{remark}
The lattice-free polytope $G$ determines a $\Q$-factorial, Gorenstein and terminal toric singularity $U = \A^d/N(G)$. More specifically, the action of $N(G)$ may be described as follows: if an element $v \in N(G)$ is represented by a lattice point $\sum_{i = 1}^{d} \alpha_{i} (v_{i}, 1)$
with $0 \le \alpha_{i} < 1$, then $v$ acts on $\A^{d}$ via co-ordinatewise multiplication by $(e^{2 \pi i \alpha_{1}}, \ldots, e^{2 \pi i \alpha_{d}})$. Moreover, every $\Q$-factorial, Gorenstein and terminal toric singularity arises from a lattice-free simplex in this way \cite{FulIntroduction}.
\end{remark}

If $u: N \times \Z  \rightarrow \Z$ denotes projection onto the second co-ordinate and $v \in N(G)$ is represented by a lattice point $\sum_{i = 1}^{d} \alpha_{i} (v_{i}, 1)$
with $0 \le \alpha_{i} < 1$, then the \emph{age} of $v$ is defined to be $u(v) = \sum_{i = 1}^{d} \alpha_{i} \in \N$. Our goal will be to use Kemperman and Scherk's theorem to put constraints on the distribution of the ages of the lattice points in $N(G)$. 
We will often use the following observation in our calculations.

\begin{remark}\label{helpful1}
If $v =  \sum_{j = 1}^{r} \alpha_j 
(v_{i_j}, 1)$ is a lattice point in $\BOX(F)$ as above, then its corresponding inverse in $N(G)$ is represented by $-v = \sum_{j = 1}^{r} (1 - \alpha_j) 
(v_{i_j}, 1)$, and $u(v) + u(-v) = r = \dim F + 1$, where the dimension of the empty face is $-1$. 
\end{remark}

We define 
\[
N(G,k,l) = \{ v \in N(G) \mid u(v) = k + 2, u(-v) = d - 2 - l \},
\]
for $0 \le k \le l \le d - 4$, and set $N(G,k,l)$ to be empty otherwise. 
Observe that since $G$ is a lattice-free simplex, $u(v) \ge 2$ for all non-zero $v \in N(G)$, and hence 
$N(G) \smallsetminus \{0\} = \coprod_{k,l} N(G,k,l)$.

\begin{remark}\label{negative}
It follows from the definition that $- N(G,k,l) = N(G,d - 4 - l, d - 4 - k)$. 
\end{remark}

\begin{remark}\label{helpful2}
By Remark \ref{helpful1}, if $v  = \sum_{i = 1}^{d} \alpha_i (v_i,1) \in N(G,k,l)$, then exactly $l - k$ of the coefficients $\alpha_i$ are zero. 
\end{remark}

The following lemma will be useful for our calculations. 

\begin{lemma}\label{flight} With the notation above, 
\[ (N(G,k,l) + N(G,m,n)) \smallsetminus \{ 0 \} \,  \subseteq \,  \coprod_{p = 0}^{k + m + 2} \;  \coprod_{q = 0}^{\min(l + m, k + n) + 2} N(G, p,q). \]
\end{lemma}
\begin{proof}
Consider a non-zero element $w = v + v'$ in $N(G)$, for some $v \in N(G, k,l)$ and $v' \in N(G,m,n)$, so that $w$ lies in $N(G, p, q)$ for some $p \le q$. 
By Remark \ref{helpful2}, if $v  = \sum_{i = 1}^{d} \alpha_i (v_i,1)$ and $v'  = \sum_{i = 1}^{d} \alpha_i' (v_i,1)$,  then exactly $l - k$ of the coefficients $\alpha_i$ are zero and  then exactly $m - n$ of the coefficients $\alpha_i'$ are zero. 
If $\{ x \}$ denotes the fractional part of a real number $x$, then $w = \sum_{i = 1}^{d} \{ \alpha_{i} + \beta_{i} \} (v_{i}, 1)$ and $u(w) = p + 2 \le  \sum_{i = 1}^{d} \alpha_{i} + \beta_{i} = k + m + 4$. 
We observe that at most $\min(l - k, n - m)$ of the coefficients $\alpha_{i} + \beta_{i}$ are zero and 
at most $k + m + 4- u(w)$ of the coefficients $\alpha_{i} + \beta_{i}$ equal $1$. Since exactly $q - p$ of the coefficients $\{ \alpha_{i} + \beta_{i} \}$ equal zero, we conclude that $q - p \le \min(l - k,  n - m) + k + m + 4 - (p + 2)$ and hence $q \le \min(l + m, k + n)+ 2$. 
\end{proof}

The following lemma will play a key role in the proof of Theorem~\ref{superA} in the 
succeeding section.  

\begin{lemma}\label{keyD}
With the notation above, let $0 \le r \le r'$ and $ d \ge  2r' + r + 7$. If $0 \le i \le r$ and $0 \le j \le r+ r' - i$, then 
\[
\sum_{k = 0}^{i} \, \sum_{l = 0}^{i + j - k} |N(G, k,l)|  \le 
\sum_{q = 0}^{i + j + 1} |N(G, r' + 1, r' + 1 + q)| + 
\sum_{p = 0}^{i} \, \sum_{q = 0}^{i + j - p} |N(G, r' + 2 + p, r' + 2 + q)|. \]
\end{lemma}
\begin{proof}
If $A$ and $B$ are the subsets of $N(G)$ containing the origin and defined by
\[ A \smallsetminus \{0\} = \coprod_{k = 0}^{i} \; \coprod_{l = 0}^{i + j - k} N(G,k,l), \]
\[
B \smallsetminus \{ 0 \} = \coprod_{m = 0}^{r'} \; \coprod_{n = 0}^{r' + i + j + 2} N(G, m,n),  
\]
then repeated application of Lemma~\ref{flight} implies that 
\[
(A + B) \smallsetminus \{ 0 \} \subseteq \coprod_{p = 0}^{r' + 1} \; \coprod_{q = 0}^{r' + i + j + 2} N(G,p,q)
\cup \coprod_{p = r' + 2}^{i + r' + 2} \; \coprod_{q = 0}^{2r' + i + j + 4  - p} N(G,p,q).
\]
By Remark \ref{negative}, 
\[ -B \smallsetminus \{ 0 \} = \coprod_{m = 0}^{r'} \; \coprod_{n = 0}^{r' + i + j + 2} N(G, d - 4 - n, d - 4 - m). \]
Consider an element $v \in N(G,p,q)$. By the above comment, if $v \in -B$, then $q \ge d - 4 - r'$, and if $v \in A$, then $q \le i + j \le r + r'$. 
We conclude that 
$A \cap (-B) = \{ 0 \}$ provided that $d - 4 - r' > r + r'$. 
Since $d \ge 2r' + r + 7$, 
Theorem~\ref{useful} applies and, after simplification, gives the result. 
\end{proof}

\begin{remark}
Observe from the proof that the bound $d \ge 2r' + r + 7$ could be replaced by $d \ge 2r' + r + 5$. However, this improved bound does not lead to new inequalities. 
\end{remark}

\section{Proof of Theorem~\ref{superA}}\label{proof1}

The goal of this section is to prove Theorem~\ref{superA} and Corollary~\ref{corA}. We will break up the proof into a sequence of 
lemmas, and use the notation of the Section~\ref{decom}. 

Recall that $\mathcal{S}$ is a regular, lattice triangulation of $\partial P$ into lattice-free simplices. 
We define 
\[
a(k,l) = \sum_{F \in \mathcal{S}} h_{F}(1) | \{ v \in \BOX(F) \mid (u(v), u(-v) ) = (2 + k, d - 2 - l)\}|, 
\]
for $2 \le k \le l \le d - 2$, and set $a(k,l) = 0$ otherwise. In order to apply our results from the previous section, we use the following easy lemma. 

\begin{lemma}\label{silly}
With the notation of the previous section, 
 \[ a(k,l) =  \sum_{ \dim G = d - 1, G \in \mathcal{S}} N(G,k,l). \] 
\end{lemma}
\begin{proof}
This follows from the definitions, using the observation that  $h_{F}(1)$ equals the number of maximal faces of $\mathcal{S}$  containing a face $F$. 
\end{proof} 

With this notation, the results of the previous section can be stated as follows:

\begin{lemma}\label{sum} 
With the notation above,  let $0 \le r \le r'$ and $ d \ge  2r' + r + 7$.
 If $0 \le i \le r$ and $0 \le j \le r+ r' - i$, then  
\begin{equation}\label{sum4}
\sum_{k = 0}^{i} \, \sum_{l = 0}^{i + j - k}a(k,l)   \le \sum_{q = 0}^{i + j  + 1} a(r' + 1, r' + 1 + q) + \sum_{p = 0}^{i} \; \sum_{q = 0}^{i + j - p} a(r' + 2 + p, r' + 2 + q). 
\end{equation}
\end{lemma}
\begin{proof}
The result follows by summing the inequalities in Lemma~\ref{keyD} over all maximal faces $G$ of the triangulation $\mathcal{S}$, and using Lemma~\ref{silly}.
\end{proof}

We will need the following  lemma.

\begin{lemma}\label{coke}
If $( \mu_0, \ldots, \mu_r )$ and $\beta$ 
are non-negative rational numbers, then
$\sum_{i = 0}^r \mu_i h_i \ge \beta \, \sum_{i = 0}^{r} h_i$, for every symmetric, unimodal sequence 
of non-negative integers $(h_0, \ldots, h_r)$, if and only $\mu_i + \cdots + \mu_{r - i} \ge \beta \, (r - 2i + 1)$ for $0 \le i \le \lfloor \frac{r}{2} \rfloor$. 
\end{lemma}
\begin{proof}
Setting $h_0 = \cdots = h_{i  - 1} = 0$ and $h_i = \cdots = h_{\lfloor \frac{r}{2} \rfloor} = 1$ shows the necessity of the conditions. Conversely, if the conditions hold and we set $h_{-1} = 0$, then
\[
\sum_{i = 0}^r \mu_i h_i = \sum_{ i = 0}^{\lfloor \frac{r}{2} \rfloor} (\mu_i + \cdots + \mu_{r - i})(h_i - h_{i - 1})
 \ge \beta \sum_{ i = 0}^{\lfloor \frac{r}{2} \rfloor}  (r - 2i + 1)(h_i - h_{i - 1}) = \beta \sum_{i = 0}^r h_i.
\]
\end{proof}


The following lemma almost completes the proof of the theorem.

\begin{lemma}\label{full4}
If $0 \le r \le r'$,  
$d \ge 2r' + r + 7$, and $\{ \lambda_{i} \mid 0 \le i \le  r + r' \}$ are non-negative rational numbers 
satisfying:
\begin{enumerate}
\item\label{numero1} $\lambda_i \ge 1$ \;  for \;  $0 \le i \le r$, 
\item\label{numero2}  $\lambda_i \ge \frac{r + 1}{2i + 1}$ \;  for \;  $r + 1 \le  i \le \lfloor \frac{r + r'}{2} \rfloor$,
\item\label{numero3} $\lambda_{p + i} + \cdots + \lambda_{q - i} \ge q - p - 2i + 1$ \; for \;  $0 \le p \le  r$, $p \le q \le 2r - p$, $0 \le  i \le \lfloor \frac{q - p}{2} \rfloor$,
\item\label{numero4} $\lambda_{p + i} + \cdots + \lambda_{q - i} \ge r - p + 1 - 2i\frac{r - p + 1}{q - p + 1}$ \, for \, 
$0 \le p \le r$, $2r - p + 1 \le q \le r + r' - p$, $0 \le i \le \lfloor \frac{q - p}{2} \rfloor$,
\end{enumerate}
then the following inequality holds:  
\begin{equation}\label{zeta4}
\lambda  a_1 +  \sum_{j = 0}^{r} a_{j + 2}  \le  a_{r' + 3} + 
\sum_{j = 0}^{r + r'} \lambda_j a_{r' + 4 + j}, 
\end{equation}
where $\lambda = \sum_{j = 0}^{r + r'} \lambda_j - r$. 
\end{lemma}
\begin{proof}
Recall from Lemma~\ref{nonlemma} that 
\begin{equation}\label{temp}
 a(t) = \sum_{F \in \mathcal{S}} \sum_{w \in \BOX(F)} t^{u(w)} h_{F}(t). 
 \end{equation}
It follows from the unimodality of the coefficients of $h_{\mathcal{S}}(t)$, the bound $d \ge  2r' + r + 7$, and condition \eqref{numero1}, that the coefficients of $h_{\mathcal{S}}(t)$ satisfy the desired inequality.  
Hence
it will be enough to consider the 
contributions of non-empty faces $F$ in \eqref{temp}. 

Fix a non-empty face $F$ of $\mathcal{S}$ 
and a lattice point $v \in \BOX(F)$ satisfying  $(u(v), u(-v)) = (2 + k, d - 2  - l)$. If we write  $h_{F}(t) = \sum_{i = 0}^{l - k}h_{i}t^{i}$ then  the contribution of the coefficients of  $t^{u(v)}h_{F}(t)$ to the left hand side of
\eqref{zeta4} equals $\sum_{i = 0}^{r - k } h_{i}$, and  the contribution to the right hand side is at least $\sum_{i = r' - k + 1}^{r' - k + 1 + (r - k)}h_{i}$, and we conclude that the contribution to the left hand side is at most the contribution to the right hand side provided $(l - k) - (r - k ) \ge r' - k  + 1$ or $k > r$. 
Moreover, if $l + k > 2r$, then 
$r - k < \frac{l - k}{2}$ and the unimodality of $h_{F}(t)$ implies that the contribution of the coefficients of $t^{u(v)}h_{F}(t)$ to the left hand side is at most $\frac{r - k + 1}{l - k + 1}h_{F}(1)$.
We conclude that it remains to bound the sum
\[
T := \sum_{k = 0}^r \lbrack \sum_{l = 0}^{2r - k} a(k,l) + \sum_{l = 2r - k + 1}^{r + r' - k} \frac{r - k + 1}{l - k + 1} a(k,l)   
\rbrack.
\]
In order to apply the inequalities of Lemma~\ref{sum} to $T$, we first describe a change of co-ordinates. Consider a vector space with basis $\{ x_{k,l} \mid 0 \le k \le r, 0 \le k + l \le r + r' \}$ and consider the basis $\{ z_{i,j} \mid 0 \le i \le r, i + j \le r + r' \}$, determined by $z_{i,j} = \sum_{k = 0}^{i} \sum_{l = 0}^{i + j - k} x_{k,l}$. One verifies that $x_{k,l} = z_{k,l} - z_{k,l - 1}  - z_{k - 1,l + 1} + z_{k - 1,l}$, 
where we set $z_{k,l} = 0$ unless  $0 \le k \le r$ and $k + l \le r + r'$. Hence, if we set $x_{k,l} = a(k,l)$, then $z_{i,j} = \sum_{k = 0}^{i} \sum_{l = 0}^{i + j - k} a(k,l)$, and we can write 
\[ 
T=  \sum_{i = 0}^{r } \sum_{j = 0}^{r + r' - i}  \alpha_{i,j}  \sum_{k = 0}^{i} \sum_{l = 0}^{i + j - k} a(k,l), 
\]
 for some coefficients $\alpha_{i,j}$. Lemma~\ref{sum} now applies and gives
 \[
 T \le  \sum_{i = 0}^{r } \sum_{j = 0}^{r + r' - i}  \alpha_{i,j} \; [ \sum_{q = 0}^{i  + j + 1} a(r' +1, r' + 1 + q) + 
 \sum_{p = 0}^{i} \; \sum_{q = 0}^{i + j - p} a(r' + 2 + p, r' + 2 + q) ].
 \]
We now change co-ordinates again, setting 
\[ z_{i,j} =  \sum_{q = 0}^{i  + j + 1} a(r' +1, r' + 1 + q) + 
\sum_{p = 0}^{i} \; \sum_{q = 0}^{i + j - p} a(r' + 2 + p, r' + 2 + q). \]
A quick calculation shows that 
\[ x_{k,l} = a(r' + 2 + k, r' + 2 + l),  k \ne 0 \]
\[ x_{0,l} = a(r' + 2, r' + 2 + l) + a(r' + 1, r' + 2 + l), l \ne 0 \]
\[ x_{0,0} = a(r' + 2, r' + 2) + a(r' + 1, r' + 2) + a(r' + 1,r' + 1), \]
so that our inequality becomes
\[
T \le   S:= \sum_{p = 0}^r \lbrack \sum_{q = 0}^{2r  - p} a(r' + 2 + p,r' + 2 + q) + \sum_{q = 2r - p + 1}^{r + r' - p} \frac{r - p + 1}{q - p + 1} a(r' + 2 + p,r' + 2 + q)   \rbrack
\]
\[
 + \sum_{q = 0}^{2r + 1} a(r' + 1,r' + 1 + q) + \sum_{q = 2r + 2}^{r  + r' + 1} \frac{ r + 1}{q} a(r' + 1, r' + 1 + q).
\]

It remains to establish new contributions to the right hand of \eqref{zeta4} whose sum is at least the right hand side $S$ of the above inequality.  
Consider a lattice point  $v \in \BOX(F)$, for some non-empty face $F$ in $\mathcal{S}$, 
 and  write $h_{F}(t) = \sum_{i = 0}^{q - p} h_i$. 
We will compare the corresponding contributions to $S$ and the right hand side of \eqref{zeta4}.

Firstly, suppose that $(u(v), u(-v)) = ((r' + 2 + p) + 2, d - 2 -(r' + 2 + q))$, for some $0 \le p \le r$ and $p \le q \le 2r - p$. On the one hand, by the definition of $a(r' + 2 + p, r' + 2 + q)$, $v$ corresponds to a contribution of $\sum_{i = 0}^{q - p} h_{i}$ to $S$.  On the other hand, the polynomial $t^{u(v)}h_F(t)$ in \eqref{temp} contributes $\sum_{i = 0}^{q - p} \lambda_{p + i} h_i$ to the right hand side of \eqref{zeta4}.  Since $h_F(t)$ has symmetric, unimodal coefficients by Lemma~\ref{patriots}, and the coefficients $\{ \lambda_i \}$ satisfy condition \eqref{numero3}, applying Lemma~\ref{coke} with $\beta = 1$ 
gives $\sum_{i = 0}^{q - p} \lambda_{p + i} h_i \ge \sum_{i = 0}^{q - p} h_{i}$, as desired.

Secondly, suppose that $(u(v), u(-v)) = ((r' + 2 + p) + 2, d - 2 -(r' + 2 + q))$, for some $0 \le p \le r$ and $2r - p + 1 \le q \le r + r' - p$. With the notation above, the corresponding contribution to $S$ is $\frac{r - p + 1}{q - p + 1}  \sum_{i = 0}^{q - p} h_{i}$, while the corresponding contribution to the right hand side of \eqref{zeta4} is $\sum_{i = 0}^{q - p} \lambda_{p + i} h_i$. As before, since the coefficients of $h_F(t)$ are symmetric and unimodal and the coefficients $\{ \lambda_i \}$ satisfy condition \eqref{numero3},
applying Lemma~\ref{coke} with $\beta = \frac{r - p + 1}{q - p + 1}$ implies that  
$\sum_{i = 0}^{q - p} \lambda_{p + i} h_i \ge \frac{r - p + 1}{q - p + 1}  \sum_{i = 0}^{q - p} h_{i}$, as 
desired.

Next suppose that $(u(v), u(-v)) = ((r' + 1) + 2, d - 2 -(r' + 1 + q))$, for some $0 \le q \le 2r + 1$. 
As above, the corresponding contribution to $S$ is  $\sum_{i = 0}^{q} h_i$ and the contribution to the right hand side of \eqref{zeta4} is $h_0 + \sum_{i = 0}^{q - 1} \lambda_i h_{i + 1}$.  If $q = 0$, then both contributions equal $h_0$, so we may assume that $q \ge 1$. Putting $p = 0$ and replacing $q$ with $q - 1$ in condition \eqref{numero3}, gives 
\begin{equation*}
\lambda_{i} + \cdots + \lambda_{q  - 1 - i} \ge q - 2i,
\end{equation*}
 for $1 \le q \le 2r + 1$ and $0 \le i \le \lfloor \frac{q - 1}{2} \rfloor$. Putting $i = 0$ implies that 
 $1 + \lambda_{0} + \cdots + \lambda_{q - 1} \ge q + 1$ and, for $1 \le i \le  \lfloor \frac{q - 1}{2} \rfloor$, using condition \eqref{numero1}, we obtain 
\[ \lambda_{i - 1} + \cdots + \lambda_{q - 1 - i} \ge \lambda_{i - 1} + q - 2i \ge q  - 2i + 1. \]
We conclude that the vector $(1, \lambda_0, \ldots, \lambda_{q - 1})$ satisfies the assumptions of Lemma~\ref{coke} with $\beta = 1$, and hence Lemma \ref{coke} implies that  $h_0 + \sum_{i = 0}^{q - 1} \lambda_i h_{i + 1} \geq \sum_{i = 0}^{q} h_i$, as desired. 

Finally, if $(u(v), u(-v)) = ((r' + 2) + 1, d - 2 -(r' + 1 + q))$, for some $2r + 2 \le q \le r + r' + 1$, then the contribution to $S$ is   $\frac{r + 1}{q} \sum_{i = 0}^{q} h_i$ and the contribution to the left hand side of 
\eqref{zeta4} is $h_0 + \sum_{i = 0}^{q - 1} \lambda_i h_{i + 1}$. Putting $p = 0$ and replacing $q$ with $q - 1$ in condition \eqref{numero4}, gives
\[
\lambda_{i} + \cdots + \lambda_{q - 1 - i} \ge (q - 2i) \, \frac{r + 1}{q}, 
\]
for $2r + 2 \le q \le r + r' + 1$ and $0 \le i \le \lfloor \frac{q - 1}{2} \rfloor$. Setting $i = 0$, we have
$1 + \lambda_{0} + \cdots + \lambda_{q - 1} \ge r + 2 \ge (q + 1) \frac{r + 1}{q}$, and, for $1 \le i \le  r + 1$, using condition \eqref{numero1} and the fact that $q \ge r + 1$, we obtain
\[ \lambda_{i - 1} + \cdots + \lambda_{q - 1 - i} \ge \lambda_{i - 1} + (q - 2i) \, \frac{r + 1}{q} \ge 1 + (q - 2i) \,\frac{r + 1}{q} \ge (q  - 2i + 1)\frac{r + 1}{q}. \] 
For $r + 2 \le i \le \lfloor \frac{q - 1}{2} \rfloor$, using condition \eqref{numero2} and the fact that $2i - 1 \le q$,  we obtain
\[ \lambda_{i - 1} + \cdots + \lambda_{q - 1 - i} \ge \lambda_{i - 1} + (q - 2i) \, \frac{r + 1}{q} \ge \frac{r + 1}{2i - 1} + (q - 2i) \, \frac{r + 1}{q} \ge (q  - 2i + 1)\frac{r + 1}{q}. \]
 We conclude that the vector $(1, \lambda_0, \ldots, \lambda_{q - 1})$ satisfies the assumptions of Lemma~\ref{coke} with $\beta = \frac{r + 1}{q}$, and hence Lemma \ref{coke} implies that  $h_0 + \sum_{i = 0}^{q - 1} \lambda_i h_{i + 1} \geq \frac{r + 1}{q} \sum_{i = 0}^{q} h_i$, as desired. 

\end{proof}

The following technical lemma, combined with the previous lemma, completes the proof of Theorem~\ref{superA}.

\begin{lemma}\label{hayden}
If $0 \le r \le r'$, and $\{ \lambda_{i} \mid 0 \le i \le  r + r' \}$ are non-negative rational numbers satisfying
\begin{enumerate}
\item\label{bat1} $\lambda_i \ge 1$ \;  for \;  $0 \le i \le r$, 
\item\label{bat2}  $\lambda_i \ge \frac{r + 1}{2i + 1}$ \;  for \;  $r + 1 \le  i \le \lfloor \frac{r + r'}{2} \rfloor$,
\item\label{bat3} $\lambda_{i} + \cdots + \lambda_{2r - i} \ge 2r - 2i + 1$ \; for \;  
$0 \le  i \le r - 1$,
\item\label{bat4} $\lambda_{i} + \cdots + \lambda_{r + r' - i} \ge r  + 1 - 2i\frac{r + 1}{r + r' + 1}$ \, for \, 
$r  + 1 \le i < \frac{r + r'}{2}$, 
\end{enumerate}
then 
\begin{itemize}
\item $\lambda_{p + i} + \cdots + \lambda_{q - i} \ge q - p - 2i + 1$ \; for \;  $0 \le p \le r$, $p \le q \le 2r - p$, $0 \le  i \le \lfloor \frac{q - p}{2} \rfloor$,
\item $\lambda_{p + i} + \cdots + \lambda_{q - i} \ge r - p + 1 - 2i\frac{r - p + 1}{q - p + 1}$ \, for \, 
$0 \le p \le r$, $2r - p + 1 \le q \le r + r' - p$, $0 \le i \le \lfloor \frac{q - p}{2} \rfloor$.
\end{itemize}
\end{lemma}
\begin{proof}
Let us first show that 
\begin{equation}\label{langer} \lambda_{p + i} + \cdots + \lambda_{q - i} \ge q - p - 2i + 1, \end{equation}  for $0 \le p \le r$, $p \le q \le 2r - p$ and $0 \le  i \le \lfloor \frac{q - p}{2} \rfloor$. 
If $p > 0$, then, by induction, we may assume that the result holds for $p - 1$. Replacing $p$ with $p  - 1$, $q$ by $q + 1$ and $i$ by $i + 1$ in \eqref{langer} gives 
$\lambda_{p + i} + \cdots + \lambda_{q - i} \ge q - p - 2i + 1,$
for  $q \le 2r - p$ and $0 \le  i \le \lfloor \frac{q - p}{2} \rfloor$, as desired.  

Hence it remains to consider the case when $p = 0$. That is, we need
to show that 
\begin{equation}\label{zooter}
\lambda_{i} + \cdots + \lambda_{q - i} \ge q - 2i + 1,
\end{equation}
 for $0 \le q \le 2r$ and $0 \le  i \le \lfloor \frac{q}{2} \rfloor$. Note that, if $q - i \le r$, then condition \eqref{bat1} implies that the inequality holds. Hence it remains to verify the inequality when $0 \le i \le q - r - 1$. If $q = 2r$, then this is exactly condition \eqref{bat3}. 
If $q < 2r$, then, by induction, we may assume that the result holds for $q + 1$. Replacing $q$ with $q + 1$ and $i$ with $i + 1$ in \eqref{zooter} gives
$\lambda_{i + 1} + \cdots + \lambda_{q +  i} \ge q - 2i$ for $0 \le  i \le q - r  - 1 < r - 1$. By condition \eqref{bat1}, $\lambda_i \ge 1$ and hence $\lambda_{i} + \cdots + \lambda_{q +  i} \ge q - 2i + 1$.

Let us now turn our attention to the second inequality 
\begin{equation}\label{boon}
\lambda_{p + i} + \cdots + \lambda_{q - i} \ge r - p + 1 - 2i\frac{r - p + 1}{q - p + 1},
\end{equation}
 for 
$0 \le p \le r$, $2r - p + 1 \le q \le r + r' - p$ and $0 \le i \le \lfloor \frac{q - p}{2} \rfloor$. 
 If $p > 0$, then, by induction, we may assume that the result holds for $p - 1$. Replacing $p$ with $p  - 1$, $q$ by $q + 1$ and $i$ by $i + 1$ in \eqref{boon} gives 
 \[
 \lambda_{p + i} + \cdots + \lambda_{q - i} \ge r - p + 2 - 2(i + 2)\frac{r - p + 2}{q - p + 3},
 \]
 for $2r - p + 1 \le q \le r + r' - p$ and $0 \le i \le \lfloor \frac{q - p}{2} \rfloor$. We need to show that
 \[ r - p + 2 - 2(i + 2)\frac{r - p + 2}{q - p + 3} \ge  r - p + 1 - 2i\frac{r - p + 1}{q - p + 1}. \]
 A quick calculation verifies that this is equivalent to 
\[
q^2 - 2q(r + i) + (p + 2i - 1)(2r - p  + 1) = (q - (p + 2i - 1))(q - (2r - p  + 1)) \ge 0. 
\]
Since $i \le  \lfloor \frac{q - p}{2} \rfloor$ implies that $q \ge p + 2i - 1$, and $q \ge 2r - p + 1$ by assumption, we conclude that the inequality above holds. 

It remains to consider the case when $p = 0$. That is, we need
to show that 
\begin{equation}\label{flipper}
\lambda_{i} + \cdots + \lambda_{q - i} \ge r + 1 - 2i\frac{r + 1}{q + 1}, 
\end{equation}
for $2r + 1 \le q \le r + r'$ and $0 \le i \le \lfloor \frac{q}{2} \rfloor$. 
First observe that if $i = \frac{q}{2}$, then 
$\lambda_{\frac{q}{2}} \ge r + 1 - q\frac{r + 1}{q + 1} = \frac{r + 1}{q + 1}$ by condition \eqref{bat2}, as desired. If $q < r + r'$, then, by induction, we may assume the result holds for $q +1$. Replacing $q$ with $q + 1$ and $i$ with $i + 1$ in \eqref{flipper} gives
\begin{equation}\label{armball}
\lambda_{i + 1} + \cdots + \lambda_{q - i} \ge r + 1 - 2(i + 1)\frac{r + 1}{q + 2}, 
\end{equation}
for $0 \le  i \le \frac{q - 1}{2}$. If $r + 1 \le i \le \frac{q - 1}{2}$, then condition \eqref{bat2} and \eqref{armball} imply that 
\[
\lambda_i + \lambda_{i + 1} + \cdots + \lambda_{q - i} \ge \frac{r + 1}{2i + 1} +  r + 1 - 2(i + 1)\frac{r + 1}{q + 2},
\]
and we are left with verifying that 
\[
\frac{r + 1}{2i + 1} +  r + 1 - 2(i + 1)\frac{r + 1}{q + 2} \ge r + 1 - 2i\frac{r + 1}{q + 1}. 
\]
A quick calculation shows that this is equivalent to $(q - 2i)(q - (2i + 1)) \ge 0$, which holds since $i \le \frac{q - 1}{2}$. If $0 \le i \le r$, then condition \eqref{bat1} and \eqref{armball} imply that
\[
\lambda_i + \lambda_{i + 1} + \cdots + \lambda_{q - i} \ge 1 +  r + 1 - 2(i + 1)\frac{r + 1}{q + 2},
\]
and we are left with showing that 
\[
r + 2 - 2(i + 1)\frac{r + 1}{q + 2} \ge  r + 1 - 2i\frac{r + 1}{q + 1}. 
\]
A quick calculation verifies that this is equivalent to $(q - 2r)(q + 1) + 2i(r + 1) \ge 0$, which holds since $q > 2r$. 

We conclude that it remains to consider the case when $p = 0$ and $q = r + r'$. That is, we are left with verifying that  
\begin{equation}\label{flipper2}
\lambda_{i} + \cdots + \lambda_{r + r' - i} \ge r + 1 - 2i\frac{r + 1}{r + r' + 1}, 
\end{equation}
for $0 \le i \le \frac{r + r'}{2}$. We have already verified the case when $i = \frac{r + r'}{2}$, and, if $r + 1 \le i < \frac{r + r'}{2}$, then this is precisely condition \eqref{bat4}.
If $0 \le i \le r$, then by induction we may assume the inequality holds for $i + 1$, and, using   condition \eqref{bat1}, we get 
\[
\lambda_{i} + \cdots + \lambda_{r + r' - i} \ge \lambda_i + r + 1 - 2(i + 1)\frac{r + 1}{r + r' + 1} \ge (1 - \frac{2(r + 1)}{r + r' + 1} ) +  r + 1 - 2i\frac{r + 1}{r + r' + 1}. 
\]
The fact that $2r + 1 \le q = r + r'$ implies that $1 - \frac{2(r + 1)}{r + r' + 1} \ge 0$, completing the proof. 
\end{proof}


Corollary~\ref{corA} follows from Theorem~\ref{superA} and the following lemma. 

\begin{lemma}\label{plan}
With the notation of Theorem~\ref{superA}, for $0 \le r \le r'$ and $m = \max( 2r, \lfloor  \frac{r  + r'}{2} \rfloor)$, 
if we set $\lambda_i = 1$ for $0 \le i \le m$ and $\lambda_i = 0$ for $m + 1 \le i \le r + r'$, then 
$(\lambda_0, \ldots, \lambda_{r + r'})$ lies in the rational polyhedron $Q(r,r')$. 
\end{lemma}
\begin{proof}
One checks immediately that $(\lambda_0, \ldots, \lambda_{r + r'})$ satisfies \eqref{cat1}, \eqref{cat2} and \eqref{cat3}. 
To verify \eqref{cat4}, we need to show that if $r  + 1 \le i < \frac{r + r'}{2}$, then $m - i + 1 \ge r + 1  - 2i\frac{r + 1}{r + r' + 1}$. A quick calculation shows that this holds if and only if $(r' - r - 1)i \le (m - r)(r + r' + 1)$, which follows since $2i \le r + r' + 1$ and $0 \le r' - r - 1 \le 2(m - r)$. 
\end{proof}




\section{A decomposition of the $h^*$-polynomial}\label{decomposition}

For the remainder of the paper we will assume that $P$ contains an interior lattice point. 
The goal of this section is to give an explicit description of the decomposition \eqref{niceshirt}, in this case. 
Theorem~\ref{main} provides a self-contained summary of the results of this section and will provide the framework for our work in the rest of the paper. 
 
Since $P$ contains an interior lattice point, the degree $s$  of $P$ equals $d$ and the codegree $l$ of $P$ equals $1$. Hence, Theorem~\ref{previous} implies that  $h^*(t)$ has a unique decomposition 
\[
h^*(t) = a(t) + tb(t),
\]
where $a(t) = t^d a(t^{-1})$ and $b(t) = t^{d - 1}b(t^{-1})$. Moreover, the coefficients $\{ a_i \}$ of $a(t)$ satisfy
$1 = a_0 \le a_1 \le a_i$ for $2 \le i \le d - 1$, and the coefficients $\{ b_i \}$ of $b(t)$ are non-negative.  
In this case, this decomposition of the $h^*$-polynomial is originally due to Betke and McMullen, who showed that $a(t)$ has positive integer coefficients and $b(t)$ has non-negative coefficients (Theorem 5 in \cite{BMLattice}).    
The goal of this section is to give an explicit description of the polynomials $a(t)$ and $b(t)$, and deduce Theorem~\ref{refinement} as a consequence. 

\begin{remark}
From a geometric perspective, Betke and McMullen's result follows from  Poincar\'e duality on the orbifold cohomology ring of a toric stack associated to $P$ (Remark 3.5 in \cite{YoWeightI}). 
\end{remark}

In order to describe the polynomials  $a(t)$ and $b(t)$, we first recall some facts about $h^*$-polynomials and triangulations and refer the reader to 
\cite{PayEhrhart} for more details (c.f. Section~\ref{decom}). 
Let $C$ be the cone over $P \times \{ 1 \}$ in $N_{\R} \times \R$ and let $u: N \times \R \rightarrow \R$ denote projection onto the second co-ordinate. Fix a regular, lattice  triangulation $\mathcal{T}$ of $P$ and if $F$ is a non-empty face of $\mathcal{T}$ 
with vertices $v_{1}, \ldots, v_{k}$,  then
let \[ \BOX(F) = \{  v \in N \times \Z \mid   v = \sum_{i = 1}^{k} \alpha_{i} (v_{i}, 1), \; 0 < \alpha_{i} < 1 \}, 
\] and 
let \[ B_{F}(t) = \sum_{w \in \BOX(F)} t^{u(w)}.\] 
We regard the empty face as a face of dimension $-1$ and 
set $\BOX(\emptyset) = \{ 0 \}$, so that $B_{\emptyset}(t) = 1$.

\begin{lemma}\label{sym2}
If $F$ is a face of $\mathcal{T}$, then $B_{F}(t) = t^{\dim F + 1} B_{F}(t^{-1})$. 
\end{lemma}
\begin{proof}
If $F$ is the empty face of $\mathcal{T}$, then the assertion follows since $B_{F}(t) = 1$ and $\dim F  = -1$ by definition. 
If $F$ is a non-empty face of $\mathcal{T}$ with vertices $v_{1}, \ldots, v_{k}$, then $k = \dim F + 1$ and  there is a natural involution $\iota$ on $\BOX(F)$ which sends $\sum_{i = 1}^{k} \alpha_{i} (v_{i}, 1)$ to $\sum_{i = 1}^{k} (1 - \alpha_{i}) (v_{i}, 1)$. Observe that $u(w) + u( \iota(w) ) = \dim F + 1$, for every $w \in \BOX(F)$. 
We compute 
\[
t^{\dim F + 1} B_{F}(t^{-1}) = t^{\dim F + 1} \sum_{w \in \BOX(F)} t^{-u(w)} =
\sum_{w \in \BOX(F)} t^{-u(\iota(w))} = B_{F}(t). 
\]
\end{proof}

Recall that the \emph{$h$-polynomial} of a face $F$ of $\mathcal{T}$ is defined by
\begin{equation*} \label{change}
h_{F}(t) = \sum_{F \subseteq G} t^{\dim G - \dim F} (1 - t)^{d - \dim G}.
\end{equation*}
As in Section~\ref{decom},
the following well-known lemma follows from Poincar\'e duality and the Hard Lefschetz theorem for projective toric varieties (cf.\@  Lemma 2.9 in \cite {YoInequalities}). 

\begin{lemma}\label{sym}
If $\mathcal{T}$ is a regular, lattice triangulation of $P$ and $F$ is a non-empty face of $\mathcal{T}$, then $h_{F}(t)$ is a polynomial with symmetric, unimodal, positive integer coefficients. The degree of $h_F(t)$ is equal to $d - 1 - \dim F$ if $F$ is contained in the boundary of $P$, and is equal to  $d - \dim F$ otherwise. 


\end{lemma}

The following theorem of Betke and McMullen expresses the $h^*$-polynomial as a sum of `shifted' $h$-polynomials. 

\begin{theorem}\label{BM}\cite[Theorem 1]{BMLattice}\label{bandm}
With the notation above, 
\[
h^*(t) = \sum_{F \in \mathcal{T}} B_{F}(t) h_{F}(t). 
\]
\end{theorem}

If we set \begin{equation*}
a'(t) = \sum_{F \in \mathcal{T}, \emptyset \neq F \subseteq \partial P} B_{F}(t) h_{F}(t), \: \: \: \:
b'(t) =  t^{-1} \cdot \sum_{F \in \mathcal{T}, \emptyset \neq F \nsubseteq \partial P} B_{F}(t) h_{F}(t), \end{equation*}
then Theorem~\ref{BM} implies that 
$h^*(t) = h_{\mathcal{T}}(t) + a'(t) + t b'(t)$. 
It follows from Lemma~\ref{sym}, and the fact that $t$ divides $B_{F}(t)$ unless $F$ is the empty face, that $a'(t)$ and $b'(t)$ are polynomials of degree less than or equal to $d - 1$. 

\begin{lemma}\label{deco}
With the notation above, we have a decomposition 
$h^*(t) = h_{\mathcal{T}}(t) + a'(t) + t b'(t)$, where $a'(t) = t^{d}a'(t^{-1})$ and $b'(t) = t^{d - 1}b'(t^{-1})$. 
\end{lemma}
\begin{proof}
We have seen in the previous discussion that $h^*(t) = h_{\mathcal{T}}(t) + a'(t) + t b'(t)$.
By Lemma~\ref{sym2} and Lemma~\ref{sym}, 
\[
 t^{d}a'(t^{-1}) = \sum_{F \in \mathcal{T}, \emptyset \neq F \subseteq \partial P} t^{\dim F + 1}B_{F}(t^{-1}) t^{d - \dim F -1}h_{F}(t^{-1}) = a'(t),
\]
\[
t^{d - 1} b'(t^{-1}) = t^{-1} \cdot \sum_{F \in \mathcal{T}, \emptyset \neq F \nsubseteq \partial P} t^{\dim F + 1}B_{F}(t^{-1}) t^{d - \dim F }h_{F}(t^{-1}) = b'(t).
\]
\end{proof}

We now consider a special triangulation $\mathcal{T}'$ of $P$ due to Hibi \cite{HibLower}. Let $\{ v_{1}, \ldots, v_{l} \}$
denote the interior lattice points of $P$ and fix a lattice  triangulation $\mathcal{B}$ of the boundary of $P$ with vertex set $\partial P \cap N$. 
We will define our triangulation $\mathcal{T}'$ inductively. 
Firstly, let  $\mathcal{T}(1)$ denote the lattice triangulation of $P$ with maximal simplices given by the convex hulls of the maximal faces of $\mathcal{B}$ and $v_{1}$, and observe that 
$h_{\mathcal{T}(1)} = h_{\mathcal{B}}(t)$. 
If $\Sigma(1)$ denotes the fan refinement of $C$ with cones given by the cones  over the faces of  $\mathcal{T}(1)$, then 
let $\Sigma(j)$ be the fan obtained by applying 
successive star subdivisions to the rays through $\{ (v_{1},1), \ldots, (v_{j}, 1) \}$, for $2 \leq j \leq l$. 
Observing that $\Sigma(j)$ is the fan over a regular, lattice triangulation $\mathcal{T}(j)$ of $P$,  
we set $\mathcal{T}' = \mathcal{T}(l)$ to be our distinguished triangulation with 
vertex set $P \cap N$.  

\begin{remark}
Geometrically, the cone $C$ over $P \times \{ 1 \}$ corresponds to Gorenstein toric singularity $U$ and the fan refinement $\Sigma(l)$ of $C$ 
corresponds to a simplicial, quasi-projective toric variety $X$ with terminal singularities and a crepant, projective birational morphism  $X \rightarrow U$. 
\end{remark}

We will need the following lemma due to Hibi and sketch the proof for the convenience of the reader. 

\begin{lemma}\cite{HibLower}\label{final}
There exists a regular, lattice triangulation $\mathcal{T}'$ of $P$ satisfying the following properties:
\begin{enumerate}
\item $\mathcal{T}'$ has vertex set $P \cap N$. 
\item $h_{\mathcal{T}'}(t) = \tilde{a}(t) + t \tilde{b}(t)$, where 
$\tilde{a}(t) = t^{d} \tilde{a}(t^{-1})$ has degree $d$ and unimodal, positive integer coefficients, and 
$\tilde{b}(t) = t^{d - 1} \tilde{b}(t^{-1})$ is either identically zero or has degree $d - 1$ and unimodal, positive integer coefficients. 
\end{enumerate}
\end{lemma}
\begin{proof} (Sketch)
With the notation of the previous discussion,  for $2 \leq j \leq l$, fix a $1$-dimensional face $F(j)$ of $\mathcal{T}(j)$ containing $v_{j}$ as a vertex, with corresponding $h$-vector 
$h_{F(j)}(t)$. If $\mathcal{B}$ is the restriction of $\mathcal{T}'$ to the boundary of $P$ then 
Hibi shows by induction on $j$ that $h_{\mathcal{T}'}(t) = h_{\mathcal{B}}(t) + t \cdot \sum_{j = 2}^{l} h_{F(j)}(t)$.
The result then follows from Lemma~\ref{sym}. 
\end{proof}




Recall, that a $d$-dimensional lattice polytope $G$ in $N$ is called a \emph{lattice-free simplex} if $G$ contains exactly $d + 1$ lattice points (necessarily its vertices). 
With this notation, we have proved the following refined version of Theorem~\ref{previous},
in the case when $P$ contains an interior lattice point. 

\begin{theorem}\label{main}
If $P$ is  a $d$-dimensional lattice polytope containing an interior lattice point, then there exist unique polynomials $a(t) = \sum_{i = 0}^{d} a_{i}t^{i}$ and $b(t) = \sum_{i = 0}^{d - 1}b_{i}t^{i}$ such that:
\begin{enumerate}

\item[] $a(t) = t^{d}a(t^{-1})$ has degree $d$

\item[]  $b(t) = t^{d - 1}b(t^{-1})$ has degree at most $d - 1$ 

\item[] $h^*(t) = a(t) + tb(t)$. 
\end{enumerate}
 
Moreover, there exists a regular, lattice triangulation  $\mathcal{T}'$ of $P$ 
such that each maximal face of   $\mathcal{T}'$ is a lattice-free simplex and 
$h_{\mathcal{T}'}(t) = \tilde{a}(t) + t \tilde{b}(t)$, where 
$\tilde{a}(t) = t^{d} \tilde{a}(t^{-1})$ has degree $d$ and unimodal, positive integer coefficients, and 
$\tilde{b}(t) = t^{d - 1} \tilde{b}(t^{-1})$ is either identically zero or has degree $d - 1$ and unimodal, positive integer coefficients. 

Let $u: N \times \Z \rightarrow \Z$ denote projection onto the second co-ordinate.  If $F$ is a non-empty face of $\mathcal{T}'$ with vertices $v_1, \ldots, v_r$, then
 set $\BOX(F) = \{ v \in  N \times \Z \mid v = \sum_{i = 1}^{k} \alpha_{i} (v_{i}, 1),  \; 0 < \alpha_{i} < 1 \}$
and
$B_{F}(t) = \sum_{w \in \BOX(F)} t^{u(w)}$.
The $h$-vector $h_{F}(t)$ is a polynomial with symmetric, unimodal, positive integer coefficients. The degree of $h_F(t)$ is equal to $d - 1 - \dim F$ if $F$ is contained in the boundary of $P$, and is equal to  $d - \dim F$ otherwise. 
If 
\begin{equation*}
a'(t) = \sum_{F \in \mathcal{T}', \emptyset \neq F \subseteq \partial P} B_{F}(t) h_{F}(t), \: \: \: \:
b'(t) =  t^{-1} \cdot \sum_{F \in \mathcal{T}', \emptyset \neq F \nsubseteq \partial P} B_{F}(t) h_{F}(t), \end{equation*}
then $t^2$ divides both $a'(t)$ and $tb'(t)$, and 
\[ a(t) = \tilde{a}(t) + a'(t), \: \: \:
 b(t) = \tilde{b}(t) + b'(t). \]






\end{theorem}
\begin{proof}
We have previously proved every statement above except the claim that $t^2$ divides $a'(t)$ and $tb'(t)$. This follows from  the observation that since $\mathcal{T}'$ has vertex set $P \cap N$, $u(w) \geq 2$ for every 
$w \in \BOX(F)$ and every non-empty face $F$ of $\mathcal{T}'$. 
\end{proof}

We conclude by giving the proof of Theorem~\ref{refinement}. We recall the statement for the convenience of the reader. 

\begin{theorem}
With the notation of Theorem~\ref{main},  
the coefficients of $a(t)$ and $b(t)$ satisfy:  
\[1 = a_{0} \leq a_{1} \leq a_{i} \textrm{ for } 2 \leq i \leq d - 1,\]
\[0 \leq b_{0} \leq b_{i}  \textrm{ for } 1 \leq i \leq d - 2.\] 
Equivalently, the coefficients of the $h^*$-polynomial of $P$ satisfy:
\begin{equation*}
1  = h^*_{0} \leq h^*_{d} \leq h^*_{1} 
\end{equation*}
\begin{equation*}
h^*_{2} + \cdots
+ h^*_{i + 1} \geq h^*_{d - 1} + \cdots + h^*_{d - i} \geq h^*_{1} + \cdots + h^*_{i}. 
\end{equation*}
for $i = 1, \ldots, \lfloor \frac{d - 1}{2} \rfloor$.
\end{theorem}
\begin{proof}
The first statement follows from Theorem~\ref{main}, using, on the one hand, the unimodality of $\tilde{a}(t)$ and $\tilde{b}(t)$ and, on the other hand, the fact that $t^2$ divides $a'(t)$ and $tb'(t)$.
If we substitute the expressions for the coefficients of $a(t)$ and $b(t)$ in \eqref{a} and \eqref{b} into these inequalities, we immediately obtain the second statement.  
 \end{proof}

\section{Additive number theory revisited}\label{revisit}

The goal of this section is to modify the results of Section~\ref{add} in order to provide the tools to prove 
Theorem~\ref{variant} in the subsequent section. 
We assume that $P$ contains an interior lattice point and will use the notation summarized in Theorem~\ref{main} above. 


Fix a regular,  lattice triangulation $\mathcal{T}'$ of $P$ into lattice-free simplices satisfying the properties described in Theorem~\ref{main}, and  fix a  maximal face $G$ of $\mathcal{T}'$.  
Recall from Section~\ref{add} that $N(G) = \coprod_{F \subseteq G} \BOX(F)$ has the structure of a  finite abelian group. If $\partial C_{G}$ denotes the cone over $(G \cap \partial P) \times \{1 \}$ in $N_{\R} \times \R$,  where $N_{\R} = N \otimes_{\Z} \R$, then  
define subsets
\[ N(G,k,l)^{a} = \{ v \in N(G) \mid u(v) = k + 2, u(-v) = d - 2 - l, v \in \partial C_G \}, \] 
\[ N(G,k,l)^{b} = \{ v \in N(G) \mid u(v) = k + 2, u(-v) = d - 1 - l, v \notin \partial C_G \}, \] 
for $0 \le k \le l \le d - 3$, and set $N(G,k,l)^a$ and $N(G,k,l)^b$ to be empty otherwise. 
Observe that since $G$ is a lattice-free simplex, $u(v) \ge 2$ for all $v \in N(G)$, and hence 
$N(G) \smallsetminus \{0\} = \coprod_{k,l} N(G,k,l)^a \cup N(G,k,l)^b$. 

\begin{remark}\label{negative2}
As in Remark~\ref{negative}, it follows from the definition that $- N(G,k,l)^a = N(G,d - 4 - l, d - 4 - k)^a$ and $- N(G,k,l)^b = N(G,d - 3 - l, d - 3 - k)^b$.  
\end{remark}

\begin{remark}\label{helpful22}
As in  Remark \ref{helpful2}, if $G$ has vertices $v_0, \ldots, v_d$ 
and $v  = \sum_{i = 0}^{d} \alpha_i (v_i,1) \in N(G,k,l)^{a}$, then exactly $l - k + 1$ of the coefficients $\alpha_i$ are zero. Similarly, if $v  = \sum_{i = 0}^{d} \beta_i (v_i,1) \in N(G,k,l)^{b}$, then exactly $l - k$ 
of the coefficients $\beta_i$ are zero. 
\end{remark}

\begin{remark}
Consider a non-zero element $w = v + v'$ in $N(G)$, for some non-zero elements $v, v' \in N(G)$. 
Observe that if $v, v' \in \partial C_{G}$, then $w \in \partial C_{G}$, and that if exactly one of $v$ or $v'$ lies in $\partial C_{G}$, then $w \notin \partial C_{G}$. 
\end{remark}

Using these three remarks, the proof of Lemma~\ref{flight} extends to give the following lemma.  

\begin{lemma}\label{flight2} With the notation above, 
\[ (N(G,k,l)^a + N(G,m,n)^a) \smallsetminus \{ 0 \} \subseteq \coprod_{p = 0}^{k + m + 2} \coprod_{q = 0}^{\min(l + m, k + n) + 2} N(G, p,q)^a, \]
\[ (N(G,k,l)^a + N(G,m,n)^b) \smallsetminus \{ 0 \} \subseteq \coprod_{p = 0}^{k + m + 2} \coprod_{q = 0}^{\min(l + m, k + n) + 2} N(G, p,q)^b, \]
\[ 
(N(G,k,l)^b + N(G,m,n)^b) \smallsetminus \{ 0 \} \subseteq
\coprod_{p = 0}^{k + m + 2} 
\coprod_{q = 0}^{\min(l + m, k + n) + 2} N(G, p - 1,q - 1)^a \cup N(G, p,q)^b. \]
\end{lemma}

We now derive three, rather technical, analogues of Lemma~\ref{keyD}. The proofs are all variants of the proof of Lemma~\ref{keyD}, but we include the details for the convenience of the reader. 


\begin{lemma}\label{keyD4a}
With the notation above, let $0 \le r \le r'$, $0 \le \alpha \le r + 1$ and $ d \ge  2r' + r + 7$. If $0 \le i \le r$ and $0 \le j \le r+ r' - i$, then 
\[
\sum_{k = 0}^{i} \sum_{l = 0}^{i +j - k}  |N(G, k,l)^a| +  |N(G, k - \alpha,l - \alpha)^b|  \le 
\]
\[
\sum_{q = 0}^{i + j + 1} 
 |N(G, r' + 1, r' + 1 + q)^a|  + |N(G, r'  + 1 - \alpha, r' + 1 + q - \alpha)^b|  \]
\[
+ \sum_{p = 0}^{i} \sum_{q = 0}^{i + j - p} |N(G, r' + 2 + p, r' + 2 + q)^a| +  
|N(G, r' + 2 - \alpha +  p, r' + 2 - \alpha + q)^b|.
\]
\end{lemma}
\begin{proof}
If $A$ and $B$ are the subsets of $N(G)$ containing the origin and defined by
\[ A \smallsetminus \{0\} = \coprod_{k = 0}^{i} \; \coprod_{l = 0}^{i + j - k} N(G,k,l)^a \cup N(G,k - \alpha,l - \alpha)^b, \]
\[
B \smallsetminus \{ 0 \} =   
\coprod_{m = 0}^{r'} \; \coprod_{n = 0}^{r' + i + j + 2} N(G, m ,n )^a  \cup N(G, m - \alpha, n - \alpha)^b,  
\]
then Lemma~\ref{flight} implies that 
\begin{align*}
(A + B) \smallsetminus \{ 0 \} \subseteq \coprod_{p = 0}^{r' + 1} \; &\coprod_{q = 0}^{r' + i + j + 2} 
N(G,p,q)^a \cup N(G,p - \alpha ,q - \alpha)^b \\
&\cup \coprod_{p = r' + 2}^{r' + i + 2} \; \coprod_{q = 0}^{2r' + i + j  + 4  - p} N(G,p,q)^a \cup N(G,p - \alpha,q - \alpha)^b.
\end{align*}

By Remark \ref{negative2}, 
\[ -B \smallsetminus \{ 0 \} =  \coprod_{m = 0}^{r'} \; \coprod_{n = 0}^{r' + i + j + 2} N(G, d - 4 - n,d - 4 - m)^a \cup N(G, d - 3- n + \alpha ,d - 3- m + \alpha)^b. \]
Consider an element $v \in N(G,p,q)^a$. If $v \in A$, then $q \le i + j \le r + r'$, and if $v \in -B$, then $q \ge d - 4 - r'$.  
Consider an element $v \in N(G,p,q)^b$. If $v \in A$, then $q \le i + j - \alpha \le r + r'$, and if $v \in -B$, then $q \ge d - 3 - r' + \alpha \ge d - 3 - r'$. 
We conclude that 
$A \cap (-B) = \{ 0 \}$ provided $d > 2r' + r + 4$. 
Theorem~\ref{useful} now applies and, after simplification, gives the result. 
\end{proof}

\begin{lemma}\label{keyD4b}
With the notation above, let $0 \le r \le r'$ and $ d \ge  2r' + r + 6$. If $0 \le i \le r$ and $0 \le j \le r+ r' - i$, then 
\[
\sum_{k = 0}^{i} \sum_{l = 0}^{i + j - k} |N(G, k - 1,l - 1)^a|  + |N(G, k,l)^b| \le  
\]
\[
\sum_{q = 0}^{i + j + 1} 
|N(G, r' + 1, r' + 1 + q)^b| 
+
\sum_{p = 0}^{i} \sum_{q = 0}^{i + j - p} |N(G, r' + 1 + p , r' + 1 + q)^a| +  
|N(G, r' + 2 + p, r' + 2 + q)^b|.
\]
\end{lemma}
\begin{proof}
If $A$ and $B$ are the subsets of $N(G)$ containing the origin and defined by
\[ A \smallsetminus \{0\} = \coprod_{k = 0}^{i} \; \coprod_{l = 0}^{i + j - k} N(G,k - 1,l - 1)^a \cup N(G,k,l)^b, \]
\[
B \smallsetminus \{ 0 \} =   \coprod_{m = 0}^{r' + 1} \; \coprod_{n = 0}^{r' + i + j + 2}  N(G, m - 1,n - 1)^a  \cup \coprod_{m = 0}^{r'} \; \coprod_{n = 0}^{r' + i + j + 2} N(G, m,n)^b,  
\]
then Lemma~\ref{flight} implies that 
\begin{align*}
(A + B) \smallsetminus \{ 0 \} \subseteq \coprod_{p = 0}^{r' + 1} \; &\coprod_{q = 0}^{r' + i +  j + 2} 
N(G,p - 1,q - 1)^a \cup N(G,p,q)^b \\
&\cup \coprod_{p = r' + 2}^{i + r' + 2} \; \coprod_{q = 0}^{2r' + i + j  + 4  - p} N(G,p - 1,q - 1)^a \cup N(G,p,q)^b.
\end{align*}
By Remark \ref{negative}, 
\[ -B \smallsetminus \{ 0 \} =  \coprod_{m = 0}^{r' + 1} \; \coprod_{n = 0}^{r' + i + j + 2} N(G, d - 3 - n,d - 3 - m)^a 
\cup \coprod_{m = 0}^{r'} \; \coprod_{n = 0}^{r' + i + j + 2} N(G, d - 3- n,d - 3- m)^b. \]
Consider an element $v \in N(G,p,q)^a$. If $v \in A$, then $q \le i + j - 1 \le r + r' - 1$, and if $v \in -B$, then $q \ge d - 4 - r'$. Consider an element $v \in N(G,p,q)^b$. If $v \in A$, then $q \le i + j \le r + r'$, and if $v \in -B$, then $q \ge d - 3 - r'$. We conclude that 
$A \cap (-B) = \{ 0 \}$ provided $d - 3 - r' > r+ r'$. 
Theorem~\ref{useful} now applies and, after simplification, gives the result. 
\end{proof}

\begin{lemma}\label{keyD4c}
With the notation above, let $0 \le r \le r'$,  $0 \le \alpha \le r + 1$ and $ d \ge  2r'  + r + \alpha + 6$. If $0 \le i \le r$ and $0 \le j \le r+ r' - i$, then 
\[
\sum_{k = 0}^{i} \sum_{l = 0}^{i +j - k}  |N(G, k,l)^b| \le 
\sum_{p = \alpha}^{i} \sum_{q = 0}^{i + j -p} |N(G, r' + 1 + p, r' + 1 + q)^a| \]
\[
+ \sum_{p = 0}^{\alpha} \sum_{q = 0}^{\alpha + i + j + 1} |N(G, r' + 1 + p, r' + 1 + q)^b| + 
\sum_{p = 0}^{i } \sum_{q = 0}^{i + j   - p} |N(G,r' + \alpha + 2 + p,r' + \alpha +  2 + q)^b|.
\]
\end{lemma}
\begin{proof}
If $A$ and $B$ are the subsets of $N(G)$ containing the origin and defined by
\[ A \smallsetminus \{0\} = \coprod_{k = 0}^{i} \; \coprod_{l = 0}^{i + j - k} N(G,k,l)^b, \]
\[
B \smallsetminus \{ 0 \} =   
\coprod_{m = 0}^{r' + \alpha} \; \coprod_{n = 0}^{r' + i + j + 1} N(G, m ,n )^a  \cup
\coprod_{m = 0}^{r'} \; \coprod_{n = 0}^{r' + \alpha +  i + j + 2} N(G, m, n)^b,  
\]
then Lemma~\ref{flight} implies that 
\begin{align*}
(A + B) \smallsetminus \{ 0 \} \subseteq \coprod_{p = 0}^{r'} \; &\coprod_{q = 0}^{r' + i + j} 
N(G,p,q)^a 
\cup \coprod_{p = r' + 1}^{r' + i + 1} \; \coprod_{q = 0}^{2r' + i + j  + 2  - p} N(G,p,q)^a \\
\cup \coprod_{p = 0}^{r' + \alpha + 1} \; &\coprod_{q = 0}^{r' + \alpha+ i + j + 2} 
 N(G,p,q)^b 
 \cup \coprod_{p = r' + \alpha +  2}^{r' + \alpha +  i + 2} \; \coprod_{q = 0}^{2r' + 2\alpha + i + j  + 4  - p} N(G,p,q)^b.
\end{align*}

By Remark \ref{negative2}, 
\begin{align*} -B \smallsetminus \{ 0 \} =  \coprod_{m = 0}^{r' + \alpha} \; &\coprod_{n = 0}^{r' + i + j + 1} N(G, d - 4 - n,d - 4 - m)^a \\
 &\cup \coprod_{m = 0}^{r'} \; \coprod_{n = 0}^{r' + \alpha + i + j + 2} N(G, d - 3- n,d - 3- m)^b. 
\end{align*}
First observe that $N(G,p,q)^a \cap A = \emptyset$ by definition. 
Consider an element $v \in N(G,p,q)^b$. If $v \in A$, then $q \le i + j \le r + r'$, and if $v \in -B$, then $q \ge d - 3 - r'$. 
We conclude that 
$A \cap (-B) = \{ 0 \}$ provided $d - 3 - r' > r + r'$. 
Theorem~\ref{useful} now applies and, after simplification, gives the result. 
\end{proof}

\begin{remark}
Although the bounds on dimension in the above lemmas can be improved, this 
does not lead to new inequalities. 
\end{remark}

\section{Proof of Theorem~\ref{variant}}\label{proof2}

The goal of this section is to sketch the proof of Theorem~\ref{variant}. 
The proof may be viewed as a (rather technical) modification of the proof of Theorem~\ref{superA}. 

We continue with the notation of Theorem~\ref{main} and the previous section. 
We define 
\[
a(k,l) = \sum_{F \in \mathcal{T}', F \subseteq \partial P} h_{F}(1) | \{ v \in \BOX(F) \mid (u(v), u(-v) ) = (2 + k, d - 2 - l)\}|, 
\]
\[
b(k,l) = \sum_{F \in \mathcal{T}', , F \nsubseteq \partial P} h_{F}(1) | \{ v \in \BOX(F) \mid (u(v), u(-v) ) = (2 + k, d - 1 - l)\}|. 
\]
for $2 \le k \le l \le d - 1$, and set $a(k,l) = b(k,l) = 0$ otherwise. As in Section~\ref{proof1}, we have the following simple lemma. 

\begin{lemma}
With the notation of the previous section, 
 \[ a(k,l) =  \sum_{ \dim G = d, G \in \mathcal{T}'} N(G,k,l)^a, \] 
\[b(k,l) =  \sum_{ \dim G = d, G \in \mathcal{T}'} N(G,k,l)^b. \] 
\end{lemma}
\begin{proof}
This follows from the definitions, using the observation that  $h_{F}(1)$ equals the number of maximal faces of $\mathcal{T}'$  containing $F$. 
\end{proof} 

 Using the lemma above, and summing the inequalities in Lemma~\ref{keyD4a},
 Lemma~\ref{keyD4b} and Lemma~\ref{keyD4c}
  over all maximal faces $G$ of the triangulation $\mathcal{T}'$, we obtain the following lemma.

\begin{lemma}\label{christmas}
With the notation above, let $0 \le r \le r'$ and  $0 \le \alpha \le r + 1$. If $0 \le i \le r$ and $0 \le j \le r+ r' - i$, then 

\begin{enumerate}

\item\label{dot1}
For $ d \ge  2r' + r + 7$,
\[ \sum_{k = 0}^{i} \sum_{l = 0}^{i +j - k}  a(k,l) +  b(k - \alpha,l - \alpha)  \le 
\sum_{q = 0}^{i + j + 1}
 a(r' + 1, r' + 1 + q)  + b(r'  + 1 - \alpha, r' + 1 + q - \alpha)  \]
\[
+ \sum_{p = 0}^{i} \sum_{q = 0}^{i + j - p} a(r' + 2 + p, r' + 2 + q) +  
b(r' + 2 - \alpha +  p, r' + 2 - \alpha + q) \]

\item\label{dot2}
For $ d \ge  2r' + r + 6$, 
\[
\sum_{k = 0}^{i} \sum_{l = 0}^{i + j - k} a(k - 1,l - 1)  + b(k,l) \le  \sum_{q = 0}^{i + j + 1}  
b(r' + 1, r' + 1 + q) 
\]
\[
+
\sum_{p = 0}^{i} \sum_{q = 0}^{i + j - p} a(r' + 1 + p , r' + 1 + q) +  
b(r' + 2 + p, r' + 2 + q).
\]

\item\label{dot3}
For  $ d \ge  2r'  + r + \alpha +  6$, 
\[
\sum_{k = 0}^{i} \sum_{l = 0}^{i +j - k}  b(k,l) \le 
\sum_{p = \alpha}^{i} \sum_{q = 0}^{i + j -p} a(r' + 1 + p, r' + 1 + q) \]
\[
+ \sum_{p = 0}^{\alpha} \sum_{q = 0}^{\alpha + i + j + 1} b(r' + 1 + p, r' + 1 + q) + 
\sum_{p = 0}^{i } \sum_{q = 0}^{i + j   - p} b(r' + \alpha + 2 + p,r' + \alpha +  2 + q).
\]

\end{enumerate}
\end{lemma}

Our next goal is to modify the proof of Lemma~\ref{full4} in order to complete the proof of the theorem. 
Recall from Theorem~\ref{main} that we may write $a(t) = \tilde{a}(t) + a'(t)$ and $b(t) = \tilde{b}(t) + b'(t)$, where $\tilde{a}(t)$ and $\tilde{b}(t)$ have symmetric, unimodal coefficients. It follows from the bounds on $d$ that all inequalities in Theorem~\ref{variant} hold for the coefficients of $\tilde{a}(t)$ and $\tilde{b}(t)$. Hence it remains to consider the contributions of 
\[
a'(t) = \sum_{F \in \mathcal{T}', \emptyset \neq F \subseteq \partial P} \; \sum_{w \in \BOX(F)} t^{u(w)} h_{F}(t),
\]
and 
\[
b'(t) =  t^{-1} \cdot \sum_{F \in \mathcal{T}', \emptyset \neq F \nsubseteq \partial P} \; \sum_{w \in \BOX(F)} t^{u(w)} h_{F}(t).
\]

We will prove the three statements in Theorem~\ref{variant} separately. 
All three proofs follow the proof of Theorem~\ref{superA}, with the third inequality requiring a little more work than the other two. 

Firstly, we consider the inequality
\begin{equation}\label{hoho}
\lambda  a_1 + \mu b_0 +  \sum_{j = 0}^{r} a_{j + 2}  + \sum_{j = 0}^{r - \alpha} b_{j + 1}  \le  a_{r' + 3} + 
\sum_{j = 0}^{r + r'} \lambda_j a_{r' + 4 + j} +  b_{r' + 2 - \alpha} + \sum_{j = 0}^{r + r'} \mu_j b_{r' + 3 - \alpha + j}, 
\end{equation}
for $d \ge 2r' + r + 7$ and  $0 \le \alpha \le r$. The first part of the proof of Lemma~\ref{full4} goes through unchanged. That is, the unimodality  and symmetry of the polynomials $h_F(t)$ implies that the contribution of $a'(t)$ and $b'(t)$ to the left hand side of \eqref{hoho} is at most the contribution of the right hand side of \eqref{hoho} provided $l + k \ge r + r' + 1$ or $k > r$, and $l + k \ge r + r' - 2\alpha + 1$ or $k > r - \alpha$, respectively. Moreover, the proof of Lemma~\ref{full4} implies that it remains to bound the following sum in terms of contributions to the right hand side of \eqref{hoho}:
\[
T := \sum_{k = 0}^r \lbrack \sum_{l = 0}^{2r - k} a(k,l) + b(k - \alpha, l - \alpha) + \sum_{l = 2r - k + 1}^{r + r' - k} \frac{r - k + 1}{l - k + 1} (a(k,l) + b(k - \alpha, l - \alpha))  
\rbrack.
\]
Applying the change of basis in the proof of Lemma~\ref{full4} with $x_{k,l} = a(k,l) - b(k - \alpha, l - \alpha)$, we may write 
\[ 
T=  \sum_{i = 0}^{r } \sum_{j = 0}^{r + r' - i}  \alpha_{i,j}  \sum_{k = 0}^{i} \sum_{l = 0}^{i + j - k} a(k,l) + b(k - \alpha, l - \alpha),
\]
 and then \eqref{dot1} in Lemma~\ref{christmas} implies that 
 \[
 T \le  \sum_{i = 0}^{r } \sum_{j = 0}^{r + r' - i}  \alpha_{i,j} \; [ \sum_{q = 0}^{i + j + 1}
 a(r' + 1, r' + 1 + q)  + b(r'  + 1 - \alpha, r' + 1 + q - \alpha)  \].
\[
+ \sum_{p = 0}^{i} \sum_{q = 0}^{i + j - p} a(r' + 2 + p, r' + 2 + q) +  
b(r' + 2 - \alpha +  p, r' + 2 - \alpha + q)]. \]
Exactly as in the proof of Lemma~\ref{full4}, changing basis again gives 
\[
T \le   S:= \sum_{p = 0}^r \lbrack \sum_{q = 0}^{2r  - p} a(r' + 2 + p,r' + 2 + q) + b(r' + 2 - \alpha + p,r' + 2 - \alpha + q) \] \[ + \sum_{q = 2r - p + 1}^{r + r' - p} \frac{r - p + 1}{q - p + 1} (a(r' + 2 + p,r' + 2 + q) + b(r' + 2 - \alpha + p,r' + 2 - \alpha + q)  ) \rbrack
\]
\[
 + \sum_{q = 0}^{2r + 1} a(r' + 1,r' + 1 + q) + b(r' + 1 - \alpha, r' + 1 - \alpha + q) + \] \[
 \sum_{q = 2r + 2}^{r  + r' + 1} \frac{ r + 1}{q} a(r' + 1, r' + 1 + q) + b(r' + 1 - \alpha, r' + 1 - \alpha + q).
\]
If we now run the rest of the proof of Lemma~\ref{full4}, considering the contributions of $a(p,q)$ and $b(p,q)$ separately, the proof follows verbatim. 


Secondly, we consider the inequality 
\begin{equation}\label{hohoho}
\lambda  a_1 + \mu b_0 +  \sum_{j = 0}^{r} a_{j + 1}  + \sum_{j = 0}^{r} b_{j + 1}  \le  a_{r' + 2} + b_{r' + 2} + 
\sum_{j = 0}^{r + r'} \lambda_j a_{r' + 3 + j} + \sum_{j = 0}^{r + r'} \mu_j b_{r' + 3 + j}, 
\end{equation}
for $d \ge 2r' + r + 6$ and  $r > 0$. 
As above, we first deduce that it suffices to bound the following sum in terms of contributions to the right hand side of \eqref{hohoho}:
\[
T := \sum_{k = 0}^r \lbrack \sum_{l = 0}^{2r - k} a(k - 1,l - 1) + b(k, l) + \sum_{l = 2r - k + 1}^{r + r' - k} \frac{r - k + 1}{l - k + 1} (a(k - 1,l - 1) + b(k, l))  
\rbrack.
\]
Changing basis as above gives
\[ 
T=  \sum_{i = 0}^{r } \sum_{j = 0}^{r + r' - i}  \alpha_{i,j}  \sum_{k = 0}^{i} \sum_{l = 0}^{i + j - k} a(k - 1,l - 1) + b(k, l),
\]
and then \eqref{dot2} in Lemma~\ref{christmas} implies that 
\[ 
T \le  \sum_{i = 0}^{r } \sum_{j = 0}^{r + r' - i}  \alpha_{i,j} \; [
 \sum_{q = 0}^{i + j + 1}  
b(r' + 1, r' + 1 + q) 
\]
\[
+
\sum_{p = 0}^{i} \sum_{q = 0}^{i + j - p} a(r' + 1 + p , r' + 1 + q) +  
b(r' + 2 + p, r' + 2 + q)].
\]
As above, changing basis again gives
\[
T \le   S:= \sum_{p = 0}^r \lbrack \sum_{q = 0}^{2r  - p} a(r' + 1 + p,r' + 1 + q) + b(r' + 2 + p,r' + 2 + q) \] \[ + \sum_{q = 2r - p + 1}^{r + r' - p} \frac{r - p + 1}{q - p + 1} (a(r' + 1 + p,r' + 1 + q) + b(r' + 2 + p,r' + 2 + q)  ) \rbrack
\]
\[
 + \sum_{q = 0}^{2r + 1} b(r' + 1, r' + 1 + q) + 
 \sum_{q = 2r + 2}^{r  + r' + 1} \frac{ r + 1}{q} b(r' + 1, r' + 1 + q).
\]
If we now run the rest of the proof of Lemma~\ref{full4}, considering the contributions of $a(p,q)$ and $b(p,q)$ separately, the proof follows verbatim.

Thirdly, we consider the inequality 
\begin{equation}\label{hoho3}
\lambda  a_1 + \mu b_0  + \sum_{j = 0}^{r} b_{j + 1}  \le  \sum_{j = 0}^{\alpha} b_{r' + 2 + j} + 
\sum_{j = 0}^{r + r' - 2 \alpha} \lambda_j a_{r' + \alpha +  3 + j} + \sum_{j = 0}^{r + r'} \mu_j b_{r' + \alpha + 3 + j}, 
\end{equation}
for $d \ge 2r' + r  + \alpha + 6$ and $0 \le \alpha \le r + 1$. 
As above, we first deduce that it suffices to bound the following sum in terms of contributions to the right hand side of \eqref{hoho3}:
\[
T := \sum_{k = 0}^r \lbrack \sum_{l = 0}^{2r - k} b(k, l) + \sum_{l = 2r - k + 1}^{r + r' - k} \frac{r - k + 1}{l - k + 1} b(k, l)  
\rbrack.
\]
Changing basis as above gives
\[ 
T=  \sum_{i = 0}^{r } \sum_{j = 0}^{r + r' - i}  \alpha_{i,j}  \sum_{k = 0}^{i} \sum_{l = 0}^{i + j - k} b(k, l),
\]
and then \eqref{dot3} in Lemma~\ref{christmas} implies that 
\[ 
T \le  \sum_{i = 0}^{r } \sum_{j = 0}^{r + r' - i}  \alpha_{i,j} \; [
\sum_{p = \alpha}^{i} \sum_{q = 0}^{i + j -p} a(r' + 1 + p, r' + 1 + q) 
\]
\[
+
\sum_{p = 0}^{\alpha} \sum_{q = 0}^{\alpha + i + j + 1} b(r' + 1 + p, r' + 1 + q) + 
\sum_{p = 0}^{i } \sum_{q = 0}^{i + j   - p} b(r' + \alpha + 2 + p,r' + \alpha +  2 + q)].
\]
Changing basis again gives, after a short calculation, 
\[
T \le   S:= \sum_{p = 0}^r \lbrack \sum_{q = 0}^{2r  - p} b(r' + 2 + \alpha + p,r' + 2 + \alpha + q)  \] \[ + \sum_{q = 2r - p + 1}^{r + r' - p} \frac{r - p + 1}{q - p + 1} b(r' + 2 + \alpha + p,r' + 2 + \alpha + q  ) \rbrack
\]
\[
+ \sum_{p = \alpha}^r \lbrack \sum_{q = 0}^{2r  - p} a(r' + 1 + p,r' + 1 + q)   + \sum_{q = 2r - p + 1}^{r + r' - p} \frac{r - p + 1}{q - p + 1} a(r' + 1 +  p,r' + 1 + q  ) \rbrack.
\]
\[
 + \sum_{p = 0}^{\alpha} [\sum_{q = 0}^{2r + \alpha + 1} b(r' + 1 + p, r' + 1 + q) + 
\sum_{q = 2r + \alpha + 2}^{r  + r' + \alpha + 1} \frac{ r + 1}{q - \alpha} \; b(r' + 1 + p, r' + 1 + q)].
\]
As before, the rest of the proof of Lemma~\ref{full4} holds for the first two terms in this sum.
The second two terms can be rewritten, by replacing $p$ and $q$ with $p + \alpha$ and $q + \alpha$, respectively, as: 
\begin{align*}
\sum_{p = 0}^{r - \alpha} \lbrack \sum_{q = 0}^{2(r - \alpha)  - p} &a( r' + \alpha + 1 + p,r' + \alpha + 1 + q)   \\
 &+ \sum_{q = 2(r - \alpha) - p + 1}^{(r - \alpha) + (r' - \alpha) - p} \frac{(r - \alpha) - p + 1}{q - p + 1} a(r' + 
\alpha + 1 +  p, r' + \alpha + 1 + q  ) \rbrack.
\end{align*}
The proof of Lemma~\ref{full4} now goes through almost unchanged with $r$ replaced by $r - \alpha$ and $r'$ replaced by $r' - \alpha$,
and we are left with considering the remaining term
in $S$: 
\[
\sum_{p = 0}^{\alpha} [\sum_{q = 0}^{2r + \alpha + 1} b(r' + 1 + p, r' + 1 + q) + 
\sum_{q = 2r + \alpha + 2}^{r  + r' + \alpha + 1} \frac{ r + 1}{q - \alpha} \; b(r' + 1 + p, r' + 1 + q)].
\]
We argue as in the proof of Lemma~\ref{full4}. 
Consider a lattice point  $v \in \BOX(F)$, for some non-empty face $F \nsubseteq \partial P$ in 
$\mathcal{T}'$ 
 and  write $(u(v), u(-v)) = ((r' + 1 + p) + 2, d - 1 -(r' + 1 + q))$ and $h_{F}(t) = \sum_{i = 0}^{q - p} h_i$. 

First, suppose that $0 \le p \le \alpha$ and $q \le 2r + \alpha + 1$. 
In this case, the corresponding contribution to $S$ is  $\sum_{i = 0}^{q - p} h_i$ and the contribution to the right hand side of \eqref{hoho3} is $h_0 + \cdots + h_{\alpha - p} + \sum_{i = 0}^{q - \alpha - 1} \mu_i h_{i + \alpha - p + 1}$.  If $q \le \alpha$, then both contributions equal  $\sum_{i = 0}^{q - p} h_i$, and hence we may assume that $q > \alpha$. We want to show that the vector $(v_0, \ldots, v_{q - p - 1}) := (1,\ldots,1, \mu_0, \ldots, \mu_{q - \alpha - 1})$ satisfies the conditions of Corollary~\ref{coke} with $\beta = 1$.
Putting $p = 0$ and replacing $q$ with $q - \alpha - 1$ in condition \eqref{numero3} gives
\[
\mu_{i} + \cdots + \mu_{q - \alpha - i - 1} \ge q - \alpha - 2i,
\]
for  
$0 \le  i \le \lfloor \frac{q - \alpha - 1}{2} \rfloor \le r$.
Moreover, condition \eqref{numero1} implies that $\mu_0, \ldots, \mu_{r} \ge 1$
 and we conclude that  $v_i + \cdots + v_{q - p - 1 - i} \ge q - p - 2i$ for $0 \le  i \le \lfloor \frac{q - \alpha - 1}{2} \rfloor$. If  $ \lfloor \frac{q - \alpha - 1}{2} \rfloor <  i \le \lfloor \frac{q - p - 1}{2} \rfloor$, then one checks that 
 $q - \alpha - 1 - (\frac{q - \alpha - 1}{2}) \le r$ since $q \le 2r + \alpha + 1$, and hence condition \eqref{numero1} implies that
$v_i + \cdots + v_{q - p - 1 - i} \ge q - p - 2i$.  We conclude that Corollary~\ref{coke} applies  with $\beta = 1$ and gives
 $\sum_{i = 0}^{q - p} h_i \le h_0 + \cdots + h_{\alpha - p} + \sum_{i = 0}^{q - \alpha - 1} \mu_i h_{i + \alpha - p + 1}$, as desired.





Finally, let us consider the case when $0 \le p \le \alpha$ and $2r + \alpha + 2 \le q \le r + r' + \alpha + 1$. As above, the  contribution to $S$ is  $\frac{r + 1}{q - \alpha} \sum_{i = 0}^{q - p} h_i$ and the contribution to the right hand side of \eqref{hoho3} is $h_0 + \cdots + h_{\alpha - p} + \sum_{i = 0}^{q - \alpha - 1} \mu_i h_{i + \alpha - p + 1}$.  
Our goal is to show that the vector $(v_0, \ldots, v_{q - p - 1}) := (1,\ldots,1, \mu_0, \ldots, \mu_{q - \alpha - 1})$ satisfies the  conditions of Corollary~\ref{coke} with $\beta = \frac{r + 1}{q - \alpha}$ and conclude that $\frac{r + 1}{q - \alpha}  \sum_{i = 0}^{q - p} h_i \le h_0 + \cdots + h_{\alpha - p} + \sum_{i = 0}^{q - \alpha - 1} \mu_i h_{i + \alpha - p + 1}$.

Putting $p = 0$ and replacing $q$ with $q - \alpha - 1$ in condition \eqref{numero4} gives
\[
\mu_{i} + \cdots + \mu_{q - \alpha - i - 1} \ge (q - \alpha - 2i)\frac{r + 1}{q - \alpha},
\] 
for 
$0 \le i \le \lfloor \frac{q - \alpha - 1}{2} \rfloor$.
For $i \le \alpha - p$, condition \eqref{numero1} and the above inequality imply that
\[
\alpha - p - i + 1 + \mu_0 + \cdots + \mu_{q - \alpha - 1 - i}  \ge \alpha - p + 1 + (q - \alpha - 2i)\frac{r + 1}{q - \alpha}.
\]
One verifies that the right hand side is greater than or equal to $(q - p - 2i + 1)\frac{r + 1}{q - \alpha}$
if and only if $q \ge r + \alpha + 1$, which holds by assumption. 

For $\alpha - p + 1 \le i \le  r + 1$, using condition \eqref{numero1} we obtain
\begin{align*}
\mu_{i - \alpha + p - 1} + \cdots + \mu_{q - \alpha - 1 - i} &\ge \mu_{i - \alpha + p - 1} + \cdots + \mu_{i - 1} + 
(q - \alpha - 2i)\frac{r + 1}{q - \alpha} \\
&\ge \alpha - p + 1 + (q - \alpha - 2i)\frac{r + 1}{q - \alpha}.
\end{align*}
One verifies that the right hand side is greater than or equal to $(q - p - 2i + 1)\frac{r + 1}{q - \alpha}$
if and only if $q \ge r + \alpha + 1$, which holds by assumption. 

For  $r + 2 \le i \le  \lfloor \frac{q - \alpha - 1}{2} \rfloor$, we have $i - 1 \le \lfloor \frac{q - \alpha - 3}{2} \rfloor \le  \lfloor \frac{r + r' - 2}{2} \rfloor$ since $q \le r + r' + \alpha + 1$. 
Hence, using  condition \eqref{numero2}, we compute 
\begin{align*}
\mu_{i - \alpha + p - 1} + \cdots + \mu_{q - \alpha - 1 - i} &\ge \mu_{i - \alpha + p - 1} + \cdots + \mu_{i - 1} + 
(q - \alpha - 2i)\frac{r + 1}{q - \alpha} \\
&\ge (\alpha - p + 1) \frac{r + 1}{2i - 1} + (q - \alpha - 2i)\frac{r + 1}{q - \alpha}.
\end{align*}
One verifies that the right hand side is greater than or equal to $(q - p - 2i + 1)\frac{r + 1}{q - \alpha}$
if and only if $i \le  \frac{q - \alpha + 1}{2}$, which holds by assumption.

It remains to consider the case when  $\lfloor \frac{q - \alpha - 1}{2} \rfloor < i \le \lfloor \frac{q - p - 1}{2} \rfloor$. One verifies that $q - \alpha - 1 - \frac{q - \alpha - 1}{2} \le \frac{r + r'}{2}$ if and only if $q \le r + r' + \alpha + 1$, which holds by assumption. Hence condition \eqref{numero2} implies that
\[
\mu_{i - \alpha + p - 1} + \cdots + \mu_{q - \alpha - 1 - i} \ge \frac{r + 1}{2q - 2\alpha - 2i - 1} (q - p - 2i + 1)
\]
One verifies that the right hand side is greater than or equal to $(q - p - 2i + 1)\frac{r + 1}{q - \alpha}$
if and only if $i \ge  \frac{q - \alpha - 1}{2}$, which holds by assumption.

We conclude that 
the vector $(v_0, \ldots, v_{q - p - 1}) := (1,\ldots,1, \mu_0, \ldots, \mu_{q - \alpha - 1})$ satisfies the  conditions of Corollary~\ref{coke} with $\beta  = \frac{r + 1}{q - \alpha}$ and that $\frac{r + 1}{q - \alpha}  \sum_{i = 0}^{q - p} h_i \le h_0 + \cdots + h_{\alpha - p} + \sum_{i = 0}^{q - \alpha - 1} \mu_i h_{i + \alpha - p + 1}$.
This completes the proof of Theorem~\ref{variant}.

Corollary~\ref{dos} follows directly from Theorem~\ref{variant} and Lemma~\ref{plan}.

\section{$h^*$-polynomials in low dimension}\label{last}

In this section, we 
 consider the question of when the results of Theorem~\ref{variant} give all possible balanced inequalities for $h^*$-polynomials of lattice polytopes with an interior lattice point.

Let us fix our notation in this section. Recall that if $P$ is a $d$-dimensional lattice polytope containing an interior lattice point, then the coefficients of the $h^*$-polynomial $(h^*_0, h^*_1, \ldots , h^*_d)$ are positive integers and $h^*_0 = 1$. We set 
\[
x_i = h^*_i - 1 \textrm{ for } 1 \le i \le d,
\]
and let $C(d) \subseteq \R^d$ denote the cone generated by all vectors of the form $(x_1, \ldots, x_d)$, for some $d$-dimensional  lattice polytope $P$ containing an interior lattice point. Recall that an inequality $\sum_{i = 0}^{d} \beta_i h^*_i \ge 0$ is \emph{balanced} if $\sum_{i = 0}^{d} \beta_i = 0$, and that all known inequalities are balanced. 
Observe that, with this notation, determining all balanced inequalities in dimension $d$ is equivalent to describing the facets of $C(d)$. 

We now consider a class of examples used by Payne in  \cite{PayEhrhart}. 
Consider positive integers $\alpha_0 \ge \alpha_{1} \ge \cdots \ge \alpha_d$ with no common factor and let $N = \Z^{d + 1}/(\sum_{i = 0}^{d} \alpha_{i} e_{i} = 0)$, where $e_0, \ldots, e_d$ denotes the standard basis of $\Z^{d + 1}$.  Observe that $N$ is a  lattice  of 
rank $d$ and, if 
$P(\alpha_{0}, \ldots, \alpha_{d})$ denotes the convex hull of the images of  $e_{0}, \ldots, e_{d}$, then $P(\alpha_{0}, \ldots, \alpha_{d})$ is a lattice polytope containing the origin in its interior.
The key point is that one can explicitly compute the $h^*$-polynomial of $P(\alpha_{0}, \ldots, \alpha_{d})$. The following lemma is implicit in \cite{PayEhrhart}. 

\begin{lemma}\label{compute}
With the notation above, the Ehrhart $h^*$-polynomial $h^*(\alpha_0, \ldots, \alpha_d)(t)$ of the lattice polytope  
$P(\alpha_{0}, \ldots, \alpha_{d})$ is given by the formula:
\[
h^*(\alpha_0, \ldots, \alpha_d)(t) = \sum_{i = 0}^{d} \sum_{j = 0}^{\alpha_i - 1}
t^{     \lceil   \sum_{0 \le k \le d, k \ne i}     \frac{j \alpha_k}{\alpha_i} - \lfloor \frac{j \alpha_k}{\alpha_i} \rfloor            \rceil +  \sum_{k = i + 1}^d \phi( \frac{j \alpha_k}{\alpha_i})     },
\]
where $\phi(x) = 1$ if $x$ is an integer and $\phi(x) = 0$, otherwise. 
\end{lemma}
\begin{proof}
Let $\mathcal{T}$ denote the triangulation of $P = P(\alpha_{0}, \ldots, \alpha_{d})$ with maximal faces $G_i$ given by 
the convex hull of the origin and $e_0, \ldots, \hat{e}_i, \ldots , e_d$, for $0 \le i \le d$. We will use Betke and McMullen's formula \eqref{bandm} to compute the Ehrhart $h^*$-polynomial of $P$. 

Observe that the finite group $N(G_i) = N \times \Z /(\Z e_0 + \cdots + \Z \hat{e}_i + \cdots + \Z e_d + \Z e_{d + 1})$ is isomorphic to the cyclic group $\Z/\alpha_i \Z$, where $e_{d + 1}$ denotes the co-ordinate of the second factor in $N \times \Z$.
The elements in $\coprod_{F \subseteq G_i} \BOX(F)$ are given by 
\[
v_{i,j} = (- je_i, \gamma_{i,j} ) = \sum_{0 \le k \le d, k \ne i} (\frac{j \alpha_k}{\alpha_i} - \lfloor \frac{j \alpha_k}{\alpha_i} \rfloor)   
(e_k, 1)  + \beta_{i,j}e_{d + 1} \textrm{ for } 0 \le j \le \alpha_i - 1,
\] 
for some $0 \le \beta_{i,j} < 1$, and where 
\[
u(v_{i,j}) = \gamma_{i,j} = \lceil   \sum_{0 \le k \le d, k \ne i}     \frac{j \alpha_k}{\alpha_i} - \lfloor \frac{j \alpha_k}{\alpha_i} \rfloor            \rceil. 
\]
If $F$ is a face of $\mathcal{T}$, then $h_F(t) = 1 + t + \cdots + t^{d - 1 - \dim F}$ if $F \subseteq \partial P$, and 
$h_F(t) = 1 + t + \cdots + t^{d - \dim F}$ otherwise. Moreover, $h_F (1)$ is equal to the number of maximal cones $G_i$ containing $F$. 
Hence, for a fixed $v \in \BOX(F)$, to determine the contribution of $t^{u(v)} h_{F}(t)$ to 
$h^*(\alpha_0, \ldots, \alpha_d)(t)$, we need to sum the contribution of $t^{u(v)}$ above, shifted appropriately, over  every maximal face $G_i$ containing $v$.  The formula now follows easily. 



\end{proof}

\begin{example}
If $\alpha_0 = \cdots = \alpha_d$, then $P(\alpha_{0}, \ldots, \alpha_{d}) \subseteq \R^d$ is the standard reflexive simplex with vertices $e_1, \ldots, e_d$ and $-e_1 - \cdots - e_d$, and $h^*(\alpha_0, \ldots, \alpha_d)(t) = 1 + t + \cdots + t^d$. 
\end{example}

Let $C'(d) \subseteq \R^d$ denote the cone generated by the inequalities of Theorem~\ref{refinement}:
\[ 0 \leq x_{d} \leq x_{1}, \] 
\begin{equation*}
x_{1} + \cdots + x_{i} \le
 x_{d - 1} + \cdots + x_{d - i} \le x_{2} + \cdots
+ x_{i + 1}. 
\end{equation*}
for $i = 1, \ldots, \lfloor \frac{d - 1}{2} \rfloor$.  By Theorem~\ref{refinement}, we have the inclusion $C(d) \subseteq C'(d)$, and one easily verifies that $C'(d)$ is a smooth $d$-dimensional cone with rays through $\{ e_{i + 1} + e_{d - i - 1} \mid 1 \leq i \leq \lfloor \frac{d - 1}{2}  \rfloor \}$, 
$\{ e_{i + 1} + e_{d - i} \mid 1 \leq i \leq \lfloor \frac{d}{2} \rfloor - 1 \}$, 
$\sum_{i = 1}^{d - 1} e_{i}$ and  $\sum_{i = 1}^{d} e_{i}$. 
The following theorem is equivalent to  Theorem~\ref{all5}. 

\begin{theorem}
With the notation above, $C(d) = C'(d)$ if and only if $d \le 5$
\end{theorem}
\begin{proof}
Using Lemma~\ref{compute},  in Figure~\ref{hoot} we realize all the primitive integer vectors of the rays of $C'(d)$  as coefficients of  $h^*$-polynomials of polytopes of the form
 $P(\alpha_{0}, \ldots, \alpha_{d})$, for $d \le 5$.   This shows that $C(d) = C'(d)$ in dimension at most $5$. 
 
 On the other hand, for $d \ge 6$, we claim that the vector $e_2 + e_{d - 1}$ in $C'(d)$ cannot be realized by the coefficients of an $h^*$-polynomial. For if there existed a polytope with $h^*$-polynomial
 $h^*(t) = 1 + t + 2t^2 + t^{3} + \cdots + t^{d - 2} + 2t^{d -1} + t^d$, then
  $a(t) = 1 + t + \cdots + t^{d}$ and $b(t) = t + t^{d - 2}$, and one computes that  $a_1 + b_0 + b_1 = 2 > a_3 + b_2 + b_3 = 1$, contradicting Example~\ref{example2}.  

\begin{figure}[htb]

\begin{tabular}{ | l |  l | l |}
\hline
Polytope & Ray in $C(d)$ & 
Dimension 
\\
\hline
$P(2,1)$ & $(1)$ &  $d = 1$ \\
$P(2,1,1)$ & $(1,0)$ &  $d = 2$ \\
 $P(2,2,1)$
  & $(1,1)$ & $d = 2$ \\
  $P(2,1,1,1)$ & $(0,1,0)$ & $d = 3$ \\ 
    $P(2,2,1,1)$ & $(1,1,0)$ & $d = 3$ \\ 
     $P(2,2,2,1)$ & $(1,1,1)$ & $d = 3$ \\ 
   $P(2,1,1,1,1)$ & $(0,1,0,0)$ & $d = 4$ \\
 $P(2,2,1,1,1)$ & $(0,1,1,0)$ & $d = 4$ \\
  $P(2,2,2,1,1)$ & $(1,1,1,0)$ & $d = 4$ \\ 
    $P(2,2,2,2,1)$ & $(1,1,1,1)$ & $d = 4$ \\ 
     $P(2,1,1,1,1,1)$ & $(0,0,1,0,0)$ & $d = 5$ \\
 $P(2,2,1,1,1,1)$  & $(0,1,1,0,0)$ & $d = 5$ \\
 $P(3,1,1,1,1,1)$ & $(0,1,0,1,0)$ & $d = 5$ \\ 
  $P(2,2,2,2,1,1)$ & $(1,1,1,1,0)$ & $d = 5$ \\ 
    $P(2,2,2,2,2,1)$ & $(1,1,1,1,1)$ & $d = 5$ \\ 
\hline 
\end{tabular}

 \caption{Realization of the rays of $C(d)$ for $d \le 5$.}
        \label{hoot}
        \end{figure}

\end{proof}

\begin{remark}
The cone $C(d)$ will not be smooth for $d \ge 6$, and we expect that the cone becomes more and more complicated as $d$ increases. On the other hand, we conjecture that $C(d)$ is a $d$-dimensional, rational polyhedral cone. 
\end{remark}

In the case when $d = 6$, the inequalities of Theorem~\ref{six} define a cone $C''(6)$ with rays through the vectors $(0,0,1,0,0,0)$, $(0,0,1,1,0,0)$, $(0,1,0,1,0,0)$, $(0,1,1,0,1,0)$, $(1,1,1,1,1,0)$,  $(1,1,1,1,1,1)$ and $(0,2,1,1,2,0)$. We realize all but one of these rays in Figure~\ref{cmon}, as well as the vector $(1,3,2,2,3,1)$. 

\begin{figure}[htb]

\begin{tabular}{ | l |  l |}
\hline
Polytope & Vector in $C(6)$ 
\\
\hline
$P(2,1,1,1,1,1,1)$  & $(0,0,1,0,0,0)$\\
 $P(2,2,1,1,1,1,1)$
 & $(0,0,1,1,0,0)$  \\
 $P(3,1,1,1,1,1,1)$  &  $(0,1,0,1,0,0)$\\ 
 $P(4,1,1,1,1,1,1)$  &  $(0,1,1,0,1,0)$\\  
  $P(2,2,2,2,2,1,1)$
  & $(1,1,1,1,1,0)$ \\ 
   $P(2,2,2,2,2,2,1)$
  & $(1,1,1,1,1,1)$ \\ 
    $P(8,2,2,2,2,2,1)$ & $(1,3,2,2,3,1)$\\ 
\hline 
\end{tabular}

 \caption{Realization of vectors in $C(6)$.}
        \label{cmon}
        \end{figure}

Recall that an inequality $\sum_{i = 1}^{d} \beta_i h^*_i \ge 0$ is \emph{strictly balanced} if $\sum_{i = 1}^{d} \beta_i = 0$, and observe that determining all strictly balanced inequalities in dimension $d$ is equivalent to describing the image $\tilde{C}(d)$ of $C(d)$ under the projection $\pi: \R^{d} \rightarrow \R^d/ \R(1,\ldots, 1)$. In the case when $d = 6$, we have shown that every ray of $C''(6)$ lies in $C(6)$ except 
the ray through $(0,2,1,1,2,0)$. On the other hand, the vector $(1,3,2,2,3,1)$ is realized in Figure~\ref{cmon}, and $(1,3,2,2,3,1) = (0,2,1,1,2,0) + (1,1,1,1,1,1)$.  We conclude that $\tilde{C}(6) = \pi(C''(6))$, completing the proof of Theorem~\ref{six}. 

On the other hand, if we allow more general objects than polytopes, as in Remark~\ref{noncon}, then we can realize the ray $(0,2,1,1,2,0)$. The following example was constructed with Christian Haase and Sam Payne. Let $P$ be the standard reflexive simplex in $\R^6$ with vertices $e_1, \ldots, e_6$ and 
$-e_1 - \ldots - e_7$, and let $\mathcal{T}$ be the triangulation with maximal faces given by the convex hulls of the origin and the maximal faces of the boundary of $P$.  Let $Q$ be the polytopal complex with the same maximal faces as $\mathcal{T}$, in which we let every maximal face be a lattice polytope with respect to the usual lattice structure except the maximal face containing $e_1, \ldots, e_d$, which we regard as a lattice polytope with respect to the lattice $\Z^6 + \Z \cdot \frac{1}{7}(4,4,2,1,1,1)$. 
Theorem \eqref{bandm} holds and allows us to compute, after a short calculation, that $h^*_Q(t) = 1 + t + 3t^2 + 2t^3 + 2t^4 + 3t^5 + t^6$, as desired.

Recall from Example~\ref{reflex} that a polytope $P$ is reflexive if and only if its $h^*$-polynomial is symmetric. Hence,  to describe the balanced inequalities for the $h^*$-polynomials of reflexive polytopes in dimension $d$, one needs to compute the cone $S(d) \subseteq C(d)$ given by the intersection of the cone $C(d)$ with the 
linear hyperplanes $\{ x_d = 0 \}$ and  $\{ x_i = x_{d - i} \mid 1 \le i \le \lfloor \frac{d}{2} \rfloor \}$. In Figure~\ref{hoot2}, we explicitly describe the rays of $S(d)$ for $d \le 6$, while the facets are defined by the inequalities in Example~\ref{reflex}. 


\begin{figure}[htb]

\begin{tabular}{ | l |  l | l |}
\hline
Polytope & Ray in $C(d)$ & 
Dimension 
\\
\hline
$P(2,1,1)$ & $(1,0)$ &  $d = 2$ \\
    $P(2,2,1,1)$ & $(1,1,0)$ & $d = 3$ \\ 
$P(2,1,1,1,1)$ & $(0,1,0,0)$ & $d = 4$ \\
  $P(2,2,2,1,1)$ & $(1,1,1,0)$ & $d = 4$ \\ 
    $P(2,2,1,1,1,1)$  & $(0,1,1,0,0)$ & $d = 5$ \\
  $P(2,2,2,2,1,1)$ & $(1,1,1,1,0)$ & $d = 5$ \\ 
   $P(2,1,1,1,1,1,1)$ & $(0,0,1,0,0,0)$ & $d = 6$ \\
 $P(3,1,1,1,1,1,1)$  &  $(0,1,0,1,0,0)$  & $d = 6$  \\ 
  $P(2,2,2,2,2,1,1)$ & $(1,1,1,1,1,0)$ & $d = 6$ \\
   
\hline 
\end{tabular}

 \caption{Realization of the rays of $S(d)$ for $d \le 6$.}
        \label{hoot2}
        \end{figure}

Let $\tilde{S}(d)$ denote the image  of $S(d)$ under the projection $\pi: \R^{d} \rightarrow \R^d/ \R(1,\ldots, 1)$. As above, determining all strictly balanced inequalities for $h^*$-polynomials of reflexive polytopes in dimension $d$ is equivalent to describing the cone $\tilde{S}(d)$. If $d = 7$, then 
$\tilde{S}(7)$ is the $2$-dimensional cone with rays given by the images of  the vectors $v_1 = (0,0,1,1,0,0,0)$
and $v_2 = (1,3,2,2,3,1,0)$ in $\R^7$,  while the facets are defined by the inequalities in Example~\ref{reflex}. 
 The vectors  $v_1$ and $v_2$ are obtained from the $h^*$-polynomials of the polytopes  $P(2,2,1,1,1,1,1,1)$ and $P(10,4,1,1,1,1,1,1)$, respectively. 


\bibliographystyle{amsplain}
\bibliography{alan}

\def\cprime{$'$}
\providecommand{\bysame}{\leavevmode\hbox to3em{\hrulefill}\thinspace}
\providecommand{\MR}{\relax\ifhmode\unskip\space\fi MR }
\providecommand{\MRhref}[2]{%
  \href{http://www.ams.org/mathscinet-getitem?mr=#1}{#2}
}
\providecommand{\href}[2]{#2}
\begin{thebibliography}{10}

\bibitem{BRComputing}
Matthias Beck and Sinai Robins, \emph{Computing the continuous discretely},
  Undergraduate Texts in Mathematics, Springer, New York, 2007, Integer-point
  enumeration in polyhedra.

\bibitem{BMLattice}
Ulrich Betke and Peter McMullen, \emph{Lattice points in lattice polytopes},
  Monatsh. Math. \textbf{99} (1985), no.~4, 253--265.

\bibitem{ehrhartpolynomial}
Eug{\`e}ne Ehrhart, \emph{Sur les poly\`edres rationnels homoth\'etiques \`a
  {$n$}\ dimensions}, C. R. Acad. Sci. Paris \textbf{254} (1962), 616--618.

\bibitem{FulIntroduction}
William Fulton, \emph{Introduction to toric varieties}, Annals of Mathematics
  Studies, vol. 131, Princeton University Press, Princeton, NJ, 1993, , The
  William H. Roever Lectures in Geometry.

\bibitem{HibSome}
Takayuki Hibi, \emph{Some results on {E}hrhart polynomials of convex
  polytopes}, Discrete Math. \textbf{83} (1990), no.~1, 119--121.

\bibitem{HibDual}
\bysame, \emph{Dual polytopes of rational convex polytopes}, Combinatorica
  \textbf{12} (1992), no.~2, 237--240.

\bibitem{HibLower}
\bysame, \emph{A lower bound theorem for {E}hrhart polynomials of convex
  polytopes}, Adv. Math. \textbf{105} (1994), no.~2, 162--165.

\bibitem{KarEhrhart}
Kalle Karu, \emph{Ehrhart analogue of the {$h$}-vector}, Integer points in
  polyhedra---geometry, number theory, representation theory, algebra,
  optimization, statistics, Contemp. Math., vol. 452, Amer. Math. Soc.,
  Providence, RI, 2008, pp.~139--146.

\bibitem{KemComplexes}
Johannes Kemperman, \emph{On complexes in a semigroup}, Nederl. Akad. Wetensch.
  Proc. Ser. A. \textbf{59} (1956), 247--254.

\bibitem{KemSmall}
\bysame, \emph{On small sumsets in an abelian group}, Acta Math. \textbf{103}
  (1960), 63--88.

\bibitem{LevRestricted}
Vsevolod Lev, \emph{Restricted set addition in abelian groups: results and
  conjectures}, J. Th\'eor. Nombres Bordeaux \textbf{17} (2005), no.~1,
  181--193.

\bibitem{MSAdvanced}
Leo Moser and Peter Scherk, \emph{Advanced {P}roblems and {S}olutions:
  {S}olutions: 4466}, Amer. Math. Monthly \textbf{62} (1955), no.~1, 46--47.

\bibitem{MPEhrhart}
Mircea Musta{\c{t}}{\v{a}} and Sam Payne, \emph{Ehrhart polynomials and stringy
  {B}etti numbers}, Math. Ann. \textbf{333} (2005), no.~4, 787--795.

\bibitem{PayEhrhart}
Sam Payne, \emph{Ehrhart series and lattice triangulations}, Discrete Comput.
  Geom. \textbf{40} (2008), no.~3, 365--376.

\bibitem{ScoConvex}
Paul Scott, \emph{On convex lattice polygons}, Bull. Austral. Math. Soc.
  \textbf{15} (1976), no.~3, 395--399.

\bibitem{StaDecompositions}
Richard Stanley, \emph{Decompositions of rational convex polytopes}, Ann.
  Discrete Math. \textbf{6} (1980), 333--342, Combinatorial mathematics,
  optimal designs and their applications (Proc. Sympos. Combin. Math. and
  Optimal Design, Colorado State Univ., Fort Collins, Colo., 1978).

\bibitem{StaHilbert1}
\bysame, \emph{On the {H}ilbert function of a graded {C}ohen-{M}acaulay
  domain}, J. Pure Appl. Algebra \textbf{73} (1991), no.~3, 307--314.

\bibitem{YoWeightI}
Alan Stapledon, \emph{Weighted {E}hrhart theory and orbifold cohomology}, Adv.
  Math. \textbf{219} (2008), no.~1, 63--88.

\bibitem{YoInequalities}
\bysame, \emph{Inequalities and {E}hrhart {$\delta$}-vectors}, Trans. Amer.
  Math. Soc. \textbf{361} (2009), no.~10, 5615--5626.

\end{thebibliography}

\end{document}